\documentclass[psamsfonts, reqno]{amsart}
\newcommand\abs[1]{\left|#1\right|}
\usepackage[margin=1in]{geometry}
\usepackage{tikz-cd}

\usepackage{amssymb,amsfonts,amsmath}
\usepackage[all,arc]{xy}
\usepackage{enumerate}
\usepackage{mathrsfs}
\usepackage{mathtools}
\usepackage{thmtools}
\usepackage{enumitem}
\usepackage{bm}
\usepackage{graphicx}
\usepackage{hyperref}
\usepackage{url}
\usepackage{xcolor}
\PassOptionsToPackage{cmyk}{xcolor}

\newtheorem{thm}{Theorem}[section]

\newtheorem{prop}[thm]{Proposition}
\newtheorem*{prop*}{Proposition}
\newtheorem{lem}[thm]{Lemma}
\newtheorem{conj}[thm]{Conjecture}

\newtheorem*{thm*}{Theorem}
\newtheorem*{lem*}{Lemma}

\theoremstyle{definition}
\newtheorem{defn}[thm]{Definition}

\newtheorem*{defn*}{Definition}
\newtheorem*{fact*}{Fact}

\theoremstyle{remark}
\newtheorem{rem}[thm]{Remark}

\newtheorem*{rem*}{Remark}
\newtheorem*{rems*}{Remarks}

\declaretheoremstyle[notefont=\bfseries,notebraces={}{},%
headpunct={},postheadspace=1em]{mystyle}

\newcommand{\Q}{\mathbf{Q}}
\newcommand{\Z}{\mathbf{Z}}
\newcommand{\R}{\mathbf{R}}
\newcommand{\C}{\mathbf{C}}

\DeclareMathOperator{\res}{res}

\DeclareMathOperator{\Hom}{Hom}
\DeclareMathOperator{\tr}{tr}
\DeclareMathOperator{\Mod}{mod}

\DeclareMathOperator{\Br}{Br}
\DeclareMathOperator{\Gal}{Gal}
\DeclareMathOperator{\ord}{ord}

\DeclareMathOperator{\sgn}{sgn}

\newcommand{\SL}[2]{\mathrm{SL}_{#1}{#2}}
\newcommand{\PSL}[2]{\mathrm{PSL}_{#1}{#2}}
\newcommand{\HilbertSymbol}[3]{\ensuremath{\left({\dfrac{#1,#2}{#3}}\right)}}

\newcommand\mathspace{\mkern\medmuskip}

\renewcommand\Im{\operatorname{Im}}

\makeatletter
\let\@@pmod\pmod
\DeclareRobustCommand{\pmod}{\@ifstar\@pmods\@@pmod}
\def\@pmods#1{\mkern4mu({\operator@font mod}\mkern 6mu#1)}
\makeatother

\makeatletter
\let\c@equation\c@thm
\makeatother

\makeatletter
\newcommand*{\house}[1]{%
	\mathord{%
		\mathpalette\@house{#1}%
	}%
}
\newcommand*{\@house}[2]{%
	\dimen@=\fontdimen8 %
	\ifx#1\scriptscriptstyle\scriptscriptfont
	\else\ifx#1\scriptstyle\scriptfont
	\else\textfont\fi\fi
	3 %
	\sbox0{%
		$#1%
		\vrule width\dimen@\relax
		\overline{%
			\kern2\dimen@
			\begingroup 
			#2%
			\endgroup
			\kern2\dimen@
		}%
		\vrule width\dimen@\relax
		\mathsurround=1.5\dimen@ 
		$%
	}%
	\ht0=\dimexpr\ht0-\dimen@\relax
	\dp0=\dimexpr\dp0+2\dimen@\relax
	\vbox{%
		\kern\dimen@ 
		\copy0 %
	}%
}

\usepackage[backend=biber,style=numeric, doi=false, isbn=false, giveninits=true, url=false]{biblatex}
\renewbibmacro{in:}{}
\graphicspath{ {./Images/} }
\DeclareGraphicsExtensions{.pdf,.png}

\title{A conjecture of Chinburg-Reid-Stover for surgeries on twist knots}

\address{\newline Department of Mathematics,\newline Rice University,\newline Houston, TX 77005, USA} 
\email{nicholas.rouse@rice.edu}

\author{Nicholas Rouse}

\date{\today}
\begin{document}

\begin{abstract}
	Associated to a hyperbolic knot complement in $S^3$ is a set of prime numbers corresponding to the residue characteristics of the ramified places of the quaternion algebras obtained by Dehn surgery on the knots. Previous work by Chinburg-Reid-Stover gives conditions on the Alexander polynomial of the knot for this set to be finite. We show that there are infinitely many examples of knots for which this set is infinite, providing evidence for a conjecture of Chinburg-Reid-Stover.
\end{abstract}
\maketitle

\section{Introduction}
When $\Gamma$ is a finitely generated group, one may form its $\SL{2}{\C}$-character variety, which is an algebraic set parametrizing representations $\Gamma \rightarrow \SL{2}{\C}$. Work of Thurston and Culler-Shalen (\cite{CullerShalen}) introduced the character variety to the study of the geometry and topology of compact $3$-manifolds. More recently, Chinburg-Reid-Stover in \cite{CRS} paid particular attention to the arithmetic aspects of the component of the character variety---called the canonical component---containing the character of the faithful discrete representation in the setting that $M$ is a hyperbolic knot complement. More specifically, if we write $C$ for the canonical component and $k(C)$ for its function field, there is a canonically defined quaternion algebra, $A_{k(C)}$, defined over $k(C)$. The geometric content encoded by this object is that it specializes at a character of a hyperbolic Dehn surgery to the quaternion algebra associated to that Kleinian group. \par
It is natural to ask how these quaternion algebras vary through different surgeries on the knot. Chinburg-Reid-Stover show in \cite{CRS} that a condition on the Alexander polynomial of the knot (called condition $(\star)$ in \cite{CRS}) guarantees that there are only finitely many rational primes lying under any finite prime ramifying the specializations of this quaternion algebra. Let us write $S$ for this set of rational primes.  Let us define $S_{D} \subseteq S$ to be the set of rational primes $p$ such that there is a specialization to the character of a hyperbolic Dehn surgery such that the quaternion algebra is ramified at some prime lying above $p$ so that knots satisfying condition $(\star)$ have $S_D$ of finite cardinality. When condition $(\star)$ fails it is shown in \cite[Theorem 1.1(3)]{CRS} (using work of Harari \cite{Harari94}) that $S$ is infinite. They furthermore state as a conjecture \cite[Conjecture 6.7]{CRS} that
\begin{conj}[{\cite[Conjecture 6.7]{CRS}}] \label{conj:inf_bad}
	Let $K$ be a hyperbolic knot in $S^3$ that fails condition $(\star)$, then, in the notation above, $S=S_{D}$.
\end{conj}
The first example of a knot with infinite $S_D$ was given in \cite{SevenFour}. Our main result is to give an infinite family of twist knots that fail condition $(\star)$ and that have infinite $S_D$.
\begin{thm} \label{thm:main}
	Let $t\geq 2$, $K_t$ be a hyperbolic twist knot with $t$ half-twists. Suppose further that there exist distinct, rational, odd primes $p$ and $q$ such that
	\begin{enumerate}
		\item $pq \mid \frac{t+1}{2}$, and
		\item $t \equiv {-1}\Mod{pq}$.
	\end{enumerate} 
	Set $T$ to be the set of rational primes $l$ such that there exists a place $\mathfrak{l}$ lying above $l$ of the trace field of some hyperbolic Dehn surgery $(d,0)$ at which the canonical quaternion algebra associated to that surgery is ramified. Then $T$ is infinite.
\end{thm}
\begin{rem}
	Infinitely many twist knots are covered by the theorem. For example, one can fix any pair of distinct, rational, odd primes $p$ and $q$ and consider the arithmetic progression $\{2pq-1 + 2jpq \> | \> j \in \Z_{\geq 0}\}$.
\end{rem}
\subsection{Outline} The paper is organized as follows. We introduce some background material on character varieties, cyclotomic fields, quaternion algebras, and the Brauer group of fields and varieties in Sections \ref{sec:charVars}, \ref{sec:cyclo}, and \ref{sec:brauer}. We then give an outline of the proof of Theorem \ref{thm:main} in Section \ref{sec:sketch} without giving proofs of intermediate steps. The remaining sections contain the proofs of the needed lemmas and propositions.
\subsection{Acknowledgments} The author wishes to thank his advisor, Alan Reid, for his help and support throughout this project. The author also wishes to thank Neil Hoffman for pointing out some errors in earlier versions of this paper.
\section{Character varieties} \label{sec:charVars}
In this section we give some background on $\SL{2}{\C}$-character varieties of Kleinian groups, trace fields, quaternion algebras. We defer more details about quaternion algebras over fields and quaternion Azumaya algebras to Section \ref{sec:brauer}.

\subsection{Generalities}
We begin by recalling that, for a finitely generated group $\Gamma$, the $\SL{2}{\C}$-representation variety of $\Gamma$ is $R(\Gamma) = \Hom(\Gamma, \SL{2}{\C})$. Given a generating set $\{\gamma_i\}$, we identify a representation $\rho : \Gamma \rightarrow \SL{2}{\C}$ with $(\rho(\gamma_1),\dots,\rho(\gamma_n)) \subset \SL{2}{\C}^n \subset \C^{4n}$. Given a different choice of generators, there is a canonical isomorphism between the two algebraic sets obtained this way. Fixing an element $\gamma \in \Gamma$, we may define a map $I_{\gamma}$ on $R(\Gamma)$ that associates to a representation $\rho$ the trace of $\rho(\gamma)$. That is, $I_{\gamma} : R(\Gamma) \rightarrow \C$ is defined by $I_{\gamma}(\rho) = \tr \rho(\gamma) = \chi_{\rho}(\gamma)$. This is a regular function on the algebraic set $R(\Gamma)$, and the ring $T$ generated by all such $I_{\gamma}$ turns out to be finitely generated (see \cite[Proposition 1.4.1]{CullerShalen}). Fixing a generating set $I_{\gamma_1},\dots,I_{\gamma_m}$ for $T$, define a map $t: R(\Gamma) \rightarrow \C^m$ by $t(\rho) = (I_{\gamma_1}(\rho),\dots,I_{\gamma_m}(\rho))$. Then the $\SL{2}{\C}$-character variety of $\Gamma$ is defined to be $t(R(\Gamma)) \subset \C^m$. This is a closed algebraic set, and different choices of generators for $T$ give isomorphic algebraic sets. In the case of $\Gamma$ equal to the fundamental group of a complement of a hyperbolic knot $K$ in $S^3$, we define its \textbf{canonical component} to be the irreducible component of $X(\Gamma)$ containing the character of the discrete and faithful representation of $\pi_1(S^3 \backslash K)$. We refer the reader to \cite{CullerShalen} for more detail.\par
Let us now fix some notation that we will use for the remainder of the paper. Let $K \subset S^3$ be a two-bridge knot. The fundamental group of $S^3 \setminus K$ can be generated by two meridians, $a$ and $b$, which are conjugate in $\pi_1\left(S^3 \setminus K \right)$. Any nonabelian representation $\rho$ of the knot group then can be conjugated to be of the form
\[
\begin{aligned}
	\rho(a) &= \begin{pmatrix}
		x & 1 \\
		0 & 1/x
	\end{pmatrix} \\
	\rho(b) &= \begin{pmatrix}
		x & 0 \\
		r & 1/x
	\end{pmatrix}.
\end{aligned}
\]
When $r=0$, the representation is reducible. We use the variables
\[
\begin{aligned}
	Z &= \chi_{\rho}(a) = \chi_{\rho}(b) = x + \frac{1}{x} \\
	R &= \chi_{\rho}(ab^{-1}) = 2 - r.
\end{aligned} 
\]
\subsection{Computation of the character varieties} We now collect some facts about the $\SL{2}{\C}$-character variety of the knot $K_t$. Most of this follows from work in \cite{MPvL}, but we prove some further specifics that will be needed for later proofs. For the knots covered by Theorem \ref{thm:main} (and others), results in \cite{MPvL} show that there is only one component containing the character of an irreducible representation. As a matter of notation our knots $K_t$ are their knots $J(-t, 2)$. We will use their work to prove the following characterization of the defining polynomial for the canonical component.
\begin{prop} \label{prop:f_tFormula}
	Let $t \geq 3$ be an odd positive integer. Write $f_t(R,Z)$ for the defining polynomial of the $\SL{2}{\C}$-character variety for the knot $K_t$. Then,
	\[
	f_t(R,Z) =
	\begin{cases}
		R^t + \left(1-Z^2\right) + \sum\limits_{i=1}^{t-1}\left(a_i - b_i Z^2\right)R^i  & t\equiv 1\Mod{4} \\
		R^t - 1 + \sum\limits_{i=1}^{t-1}\left(a_i - b_i Z^2\right)R^i  & t\equiv 3\Mod{4},
	\end{cases}
	\]
	where $a_i, b_i \in \Z$ and $a_{t-1}=b_{t-1} = 1$. In particular, $\deg_Z{f_t} = 2$.
\end{prop}
The proof is straightforward using results in \cite{MPvL}, but their work requires some setup. The following definition is from {\cite[Definitions 3.1 and 3.2]{MPvL}} except their $f$ and $g$ are our $\sigma$ and $\tau$, respectively, and we use $i$ to index rather than their $j$ or $k$.
\begin{defn}
	Set $\sigma_0 = 0$, $\sigma_1 = 1$. For all other $i \in \Z$, define $\sigma_i \in \Z[u]$ by the relation $\sigma_{i+1} - u \sigma_i + \sigma_{i-1} = 0$. For all integers $i$, define $\tau_i = \sigma_i - \sigma_{i-1}$, $\Phi_{2i} = \sigma_i$, $\Phi_{2i-1} = \tau_i$, and $\Psi_i = \Phi_{i+1} - \Phi_{i-1}$.
\end{defn}
\begin{lem}[{\cite[Lemma 3.5]{MPvL}}] \label{lem:MPvLPhiRelations}
	For any integer $i$, we have $\Phi_{i+2} = u\Phi_i - \Phi_{i-2}$, $\Phi_i = (-1)^{i+1}\Phi_{-i}$, and $\deg \Phi_i = \left\lfloor{\left(\abs{i}-1\right)/2}\right\rfloor$.
\end{lem}
Using these notations, we may write down polynomials defining the $\PSL{2}{\C}$- and $\SL{2}{\C}$-character varieties.
\begin{prop}[{\cite[Proposition 3.8]{MPvL}}] \label{prop:MPvLvarieties}
	Let $\mu, \nu$ be any integers with $\nu$ even. Let $Y = \chi_{\rho}(a^2)$ and $R$, $Z$ be as in Section \ref{sec:sketch}. The $\PSL{2}{\C}$-character variety of $J(\mu, \nu)$ is isomorphic to the subvariety of $\mathbf{A}^2$ cut out by the polynomial
	\[
	h_{\mu, \nu}(R, Y) = \sigma_{\xi}(\theta)\left(\Phi_{-\mu}(R)\Phi_{\mu-1}(R)\left(Y-R\right)-1\right) + \sigma_{\xi-1}(\theta),
	\]
	where $\theta = \Phi_{-\mu}(R)\Phi_{\mu}(R)(Y-R) + 2$ and $\xi = \nu/2$. The $\SL{2}{\C}$-character variety is isomorphic to the double cover of the above model of the $\PSL{2}{\C}$-character variety given by $Y = Z^2 - 2$.
\end{prop}
Fortunately for us, the above formula cleans up significantly when we restrict attention to the family of twist knots $K_t = J(-t, 2)$. In particular, $\xi = 1$ in which case $\sigma_{\xi} = 1$ and $\sigma_{\xi-1} = 0$. This also allows us to ignore the $\theta$ variable. We summarize this as
\begin{lem} \label{lem:simplerCharVar}
	Let $t$ be a nonzero integer. The $\PSL{2}{\C}$-character variety of the knot $K_t$ is given by
	\[
	h_t(R,Y) = \Phi_t(R)\Phi_{-t-1}(R)(Y-R) - 1.
	\]
\end{lem}
\begin{proof}[Proof of Proposition \ref{prop:f_tFormula}]
	The content of Proposition \ref{prop:f_tFormula} is
	\begin{enumerate}[ref=(\arabic*)]
		\item \label{item:bidegree} the $(R,Z)$-bidegree of $f_t(R,Z)$ is $(t,2)$, \\
		\item \label{item:noZfirst} there are no $Z$ to the first power terms, \\
		\item \label{item:leadingCoefficients} the coefficients of $R^t$ and $R^{t-1}$ are $1$ and $(1-Z^2)$ respectively, and \\
		\item \label{item:constantTerm} the constant term with respect to $R$ is $(1-Z^2)$ when $t \equiv 1 \Mod{4}$ and $-1$ when $t \equiv 3 \Mod{4}$.
	\end{enumerate}
	We note that the $R$-degree part of item \ref{item:bidegree} follows from Lemma \ref{lem:MPvLPhiRelations} and the $Z$-degree part from lemma \ref{lem:simplerCharVar} as the $\Phi_i$ family are univariate polynomials in $R$. We also get item \ref{item:noZfirst} from Lemma \ref{lem:simplerCharVar}. To prove items \ref{item:leadingCoefficients} and \ref{item:constantTerm}, we first show that $\Phi_i$ is of the form
	\[
	\Phi_i(u) = 
	\begin{cases}
		0 & i = 0 \\
		u^{\frac{\abs{i}-1}{2}}-u^{\frac{\abs{i}-3}{2}} + \cdots + 1 & i \equiv 1,7\Mod{8} \\
		u^{\frac{\abs{i}-1}{2}}-u^{\frac{\abs{i}-3}{2}} + \cdots - 1 & i \equiv 3,5\Mod{8} \\
		\sgn(i)u^{\frac{\abs{i}}{2}-1} + \lambda_{i_1}u^{\frac{\abs{i}}{2}-3} + \cdots + \lambda_{i_2}u & i\neq 0, i \equiv 0,4\Mod{8} \\
		\sgn(i)u^{\frac{\abs{i}}{2}-1} + \lambda_{i_1}u^{\frac{\abs{i}}{2}-3} + \cdots + 1 & i \equiv 2\Mod{8} \\
		\sgn(i)u^{\frac{\abs{i}}{2}-1} + \lambda_{i_1}u^{\frac{\abs{i}}{2}-3} + \cdots - 1 & i \equiv 6\Mod{8},
	\end{cases}
	\]
	where $\lambda_{i_j} \in \Z$. That is, when $i$ is odd, $\Phi_i$ is monic, its second leading coefficient is $-1$, and its constant term is $\pm 1$ depending on the residue class modulo $8$ of $i$, and when $i$ is even, the leading coefficient is equal to the sign of $i$, the second leading coefficient is $0$, and the constant term is $0$ when $i$ is divisible by $4$, $1$ when $i \equiv 2 \Mod{8}$ and $-1$ when $i \equiv 6\Mod{8}$. The degrees follow from Lemma \ref{lem:MPvLPhiRelations}. Moreover, note that it suffices to handle the cases of $i$ positive since we have the relation $\Phi_i = (-1)^{i+1}\Phi_{-i}$, so we henceforth assume $i$ positive. We use the base cases $\Phi_1(u) = \Phi_2(u) = 1$, $\Phi_3(u) = u-1$, $\Phi_4(u) = u$, and $\Phi_5(u) = u^2-u-1$. We use the relation $\Phi_{i+2} = u\Phi_i - \Phi_{i-2}$ to see immediately that, when $i$ is even, $\Phi_{i+2}$ must be monic. To see that the second leading coefficient is $0$ when $i$ is odd, we use the equality $\Phi_{i+2} = u\Phi_i - \Phi_{i-2}$ to find that the second leading coefficient of $\Phi_{i+2}$ is the same as that of $u\Phi_i$ because the second leading coefficient is the coefficient of $u$ to the power $(i-2)/2$, but $\Phi_{i-2}$ has degree $(i-4)/2$, so using the base case of $\Phi_4$, we see that the second leading coefficient of $\Phi_i$ is $0$ when $i$ is even. For $i$ odd, note that the relation $\Phi_{i+2} = u\Phi_i - \Phi_{i-2}$ implies that the leading coefficient of $\Phi_{i+2}$ is the same as that of $\Phi_i$ (and so equal to $1$ by induction). Further the degree of $\Phi_{i-2}$ is $(i-3)/2$, so it makes no contribution to the coefficient of $u^{\frac{i-1}{2}}$, which is the second leading term of $\Phi_{i+2}$. Then by the inductive hypothesis, we see that the second leading coefficient must be $-1$. To treat the constant term, note that $\Phi_{i+2}(0) = -\Phi_{i-2}$, so the constant term depends only on the residue class modulo $8$. The base cases listed above combined with this observation handle all the constant terms. \par
	Now we may prove item \ref{item:leadingCoefficients}. Recall that $t$ is odd and positive, so $\Phi_t(R)$ is monic and $\Phi_{-t-1}(R)$ has leading coefficient $-1$. Note that the second leading coefficient of $\Phi_{t-1}(R)$ is $0$, so the second leading coefficient of $\Phi_t(R)\Phi_{-t-1}(R)$ is the coefficient of $-R^{\frac{t-3}{2}}$ times $-R^{\frac{t+1}{2}-1}$, namely $1$. We summarize this as
	\[
	\begin{aligned}
		h_t(R,Y) &= \Phi_t(R)\Phi_{-t-1}(R)(Y-R) - 1 \\
		&=\left(-R^{t-1}+R^{t-2}+ \cdots \right)(Y-R) - 1 \\
		&=R^t + (-1-Y)R^{t-1} + \cdots \\
	\end{aligned}
	\]
	Then we see that $h_t(R,Y)$ is monic is $R$ and after substituting $Y=Z^2-2$, the second leading coefficient is $(1-Z^2)$. \par
	For item \ref{item:constantTerm}, we first suppose that $t\equiv 1\Mod{4}$, so $t$ is equivalent to either $1$ or $5$ modulo $8$. If $t \equiv 1 \Mod{8}$, then $\Phi_t(0) = 1$ and $\Phi_{-t-1}(0) = -1$ as $-t-1 \equiv 6 \Mod{8}$. If $t \equiv 5 \Mod{8}$, then $\Phi_t(0) = -1$ and $\Phi_{-t-1}(0) = 1$ since $-t-1 \equiv 2 \Mod{8}$. In either case, then, $h_t(0,Y) = -Y-1$, which becomes $1-Z^2$. If $t\equiv 3 \Mod{4}$, then $-t-1$ is divisible by $4$, so $\Phi_{t-1}(0) = 0$, so $h_t(0,Y) = -1$.
\end{proof}
Next, we prove a lemma relating the character variety to the Alexander polynomial.
\begin{lem} \label{lem:constantTerm} Let $f_t$ be as above and $\Delta_{K_t}(x) = \left(\frac{t+1}{2}\right)x^2 - tx + \left(\frac{t+1}{2}\right)$ be the Alexander polynomial of $K_t$. Then,
	\[
	f_t(2,x+x^{-1}) = \frac{-\Delta_{K_t}(x^2)}{x^2} = -\left(\left(\dfrac{t+1}{2}\right)x^2-t+\left(\dfrac{t+1}{2}\right)x^{-2}\right).
	\]
\end{lem}
\begin{proof}
	We first note that
	\[
	\Phi_i(2) =
	\begin{cases}
		1 & i \text{ odd } \\
		i/2 & i \text { even }.
	\end{cases}
	\]
	Indeed, note that $\Phi_1(2) = \Phi_3(2) = 1$, $\Phi_0(2) = 0$, and $\Phi_2(2) = 1$. The specialized relation $\Phi_{i+2}(2) = 2\Phi_i(2) - \Phi_{i-2}(2)$ immediately handles the $i$ odd case. For $i$ even, we may induct and note that
	\[ 
	\begin{aligned}
		\Phi_{i+2}(2) &= 2\Phi_i(2) - \Phi_{i-2}(2) \\
		&= 2\left(\dfrac{i}{2}\right)-\dfrac{i-2}{2} \\
		&= \frac{i+2}{2}.
	\end{aligned} 
	\]
	Now we compute using Lemma \ref{lem:simplerCharVar}. We have $h_t(2,Y) = \left(\dfrac{-t-1}{2}\right)\left(Y-2\right)-1$. Then we substitute $Y=Z^2-2$ and $Z = x + x^{-1}$. After cleaning up, we get the desired equality.
\end{proof}
\subsection{Number Fields and Quaternion Algebras Associated to Subgroups of \texorpdfstring{$\SL{2}{\C}$}{SL2(C)}}We next turn to some background information about subgroups of $\SL{2}{\C}$. A subgroup $\Gamma$ of $\SL{2}{\C}$ is \textbf{non-elementary} if its image in $\mathrm{PSL}_2\C$ has no finite orbit in its action on $\mathbf{H}^3\cup\widehat{\C}$. Given an non-elementary subgroup $\Gamma$ of $\SL{2}{\C}$, we define its \textbf{trace field} by $k_{\Gamma} = \Q\left(\tr \gamma \mathspace | \mathspace \gamma \in \Gamma\right)$ and \textbf{quaternion algebra} by the $k_{\Gamma}$-span of elements of $\Gamma$. That is,
\[
A_{\Gamma} = \left\{\sum_{\text{finite}}\alpha_i \gamma_i \mathspace \big| \mathspace \alpha_i \in k_{\Gamma}, \gamma_i \in \Gamma \right\}.
\]
As shown in \cite[p.78]{MR}, we may write a Hilbert symbol for this quaternion algebra as
\[
\HilbertSymbol{\chi(g)^2-4}{\chi(g,h)-2}{k_{\Gamma}},
\]
where $g, h$ are noncommuting hyperbolic elements of $\Gamma$. 
In fact, this pointwise construction extends to define a quaternion algebra over the function field of the canonical component. 

\begin{prop}[{\cite[Corollary 2.9]{CRS}}] \label{CRSFunctionFieldHilbertSymbol} Let $\Gamma$ be a finitely generated group, and $C$ an irreducible component of the character variety of $\Gamma$ defined over the number field $k$. Assume that $C$ contains the character of an irreducible representation, and let $g,h \in \Gamma$ be two elements such that there exists a representation $\rho$ with character $\chi_{\rho} \in C$ for which the restriction of $\rho$ to $\langle g, h \rangle$ is irreducible. Then the canonical quaternion algebra $A_{k(C)}$ is described by the Hilbert symbol
	\[
	\HilbertSymbol{I_g^2-4}{I_{[g,h]}-2}{k(C)}.	
	\] 
\end{prop}
\section{Algebraic number theory and cyclotomic fields} \label{sec:cyclo}
In this section we collect some basic facts about number fields, cyclotomic fields, and their maximal totally real subfields that we will use in later sections. For context, the trace field of a $(d,0)$ surgery always contains the maximally totally real subfield of the $d$-th cyclotomic field as a subfield, and we often leverage properties of this subfield and its elements for information about the trace field.
\subsection{Number fields}
We first fix some language. A \textbf{number field} is a finite degree extension of the rational numbers. For a number field of degree $n$, there are $n$ distinct embeddings into $\C$ that fix $\Q$. If all of these embeddings have image inside $\R$, then the field is said to be \textbf{totally real}, and if none of them do, the field is \textbf{totally imaginary}. For $L/K$ an extension of number fields, we will use an important function $N_{L/K} : K \rightarrow K$ called the \textbf{field norm}. It is defined by
\[
	N_{L/K}(x) = \prod_{\sigma} \sigma(x),
\]
where the product is taken over embeddings $\sigma: L \rightarrow \overline{K}$ that fix $L$ element-wise and $\overline{K}$ is an algebraic closure of $K$. The norm behaves well in towers.
\begin{prop}[{\cite[Corollary I.2.7]{NeukirchANT}}]
	Let $K \subseteq L \subseteq M$ be a tower of finite field extensions. Then,
	\[
		N_{M/K} = N_{L/K} \circ N_{M/L}.
	\]
\end{prop}
For an element $x \in K$, a rational prime $l$ divides $N_{K/\Q}(x)$ if and only if there is a prime ideal $\mathfrak{l}$ of $K$ lying above $l$ such that $x \in \mathfrak{l}$. See the next subsection for a primer on ideals in number fields. One may efficiently compute norms if the minimal polynomial for an element is known. In particular, if $p(t) \in \Q[t]$ is a monic minimal polynomial for $x$, then $N_{K/\Q}(x) = (-1)^{\deg p} p(0)$. This observation is actually how we produce candidate primes that might ramify the relevant quaternion algebras.
\subsection{Rings of integers and prime ideals}
For a number field $K$, the integral closure of $\Z$ inside $K$ is called the \textbf{ring of integers} of $K$, and is often written $\mathcal{O}_K$. This ring may be more concretely identified as the set of elements of $K$ that satisfy a monic polynomial with coefficients in $\Z$. The ring of integers is often not a unique factorization domain, but its ideals uniquely factor into prime ideals. By a ``prime in a number field," we will mean a prime ideal in the ring of integers of that number field. For an extension $L/K$ of number fields and a prime $\mathfrak{l}$ of $\mathcal{O}_K$, the primes appearing in the factorization of $\mathfrak{l}\mathcal{O}_L$ are said to \textbf{lie above} $\mathfrak{l}$, and $\mathfrak{l}$ \textbf{lies below} those primes of $\mathcal{O}_L$.

For a prime $\mathfrak{l}$ of $\mathcal{O}_K$ with factorization
\[
	\mathfrak{l}\mathcal{O}_L = \prod_i \mathfrak{L_i}^{e_i},
\]
the integer $e_i$ is the \textbf{ramification index} of $\mathfrak{L_i}$ over $\mathfrak{l}$. There \textit{is} an analogy between ramification of prime ideals in extensions and ramification of quaternion algebras, but they are distinct concepts. The degree of the field extension $\mathcal{O}_L/\mathfrak{L_i}$ over $\mathcal{O}_K/\mathfrak{l}$ is the \textbf{inertia degree} of $\mathfrak{L_i}$ over $\mathfrak{l}$. A fundamental fact is that if $\mathfrak{l}$ splits into $r$ distinct primes in $L$ with inertia degrees $f_i$ and inertia degrees $e_i$, then $[L : K] = \sum\limits_{i=1}^{r}e_i f_i$. When $r=e_i=1$ for all $i$, we say that $\mathfrak{l}$ is \textbf{totally inert} or simply \textbf{inert}. When $L/K$ is Galois, the ramification indices and inertia degrees for any given prime $\mathfrak{L_i}$ above $\mathfrak{l}$ are the same as any other prime $\mathfrak{L_j}$ above $\mathfrak{l}$. In this setting, we simply write $e$ and $f$ for the ramification index and inertia degree of any prime above $\mathfrak{l}$, and we have $[L : K] = ref$, where $r$ is the number of distinct primes that $\mathfrak{l}$ splits into.
\subsection{Cyclotomic fields}
A \textbf{cyclotomic field} is a field extension obtained by adjoining roots of unity. All finite field extensions are cyclotomic. We shall write $\zeta_d$ for a primitive $d$-th root of unity, and by the $\mathbf{d}$\textbf{-th cyclotomic field}, we mean $\Q(\zeta_d)$, the field obtained by adjoining $\zeta_d$ to $\Q$. These fields and hence their subfields are abelian. For $d \geq 3$, $\Q(\zeta_d)$ is totally imaginary, but it does have a totally real subfield generated by $\zeta_d + \zeta_d^{-1}$. Moreover this subfield is maximal with respect to inclusion among totally real subfields, so it may be uniquely identified as the maximal totally real subfield of $\Q(\zeta_d)$. We will often write $2\cos(2\pi/d)$ for $\zeta_d + \zeta_d^{-1}$ without a particular embedding of $\zeta_d$ into $\C$ in mind. We will also often write $\Q(\zeta_d)^+$ in place of $\Q(2\cos(2\pi/d))$. The rings of integers of $\Q(\zeta_d)$ and $\Q(\zeta_d)^+$ are $\Z[\zeta_d]$ and $\Z[2\cos(2\pi/d)]$, respectively (see \cite[Theorem 2.6, Proposition 2.16]{WashingtonCycloFields}).  \par
We will use the following description of the splittings of primes in cyclotomic extensions.
\begin{thm}[{\cite[Proposition I.10.3]{NeukirchANT}}] \label{thm:cyclosplitting} Let $l$ be a rational prime comprime to $d$. Then $l$ factors into distinct primes in $\Z[\zeta_d]$ all with inertia degree equal to the multiplicative order of $l \Mod{d}$.
\end{thm}
We emphasize some special cases of the above theorem. A rational prime $l$ is totally split in $\Q(\zeta_d)$ if and only if it is $1 \Mod{d}$ and is totally inert if and only if it is a primitive root modulo $d$. It follows that if $l \in \Z$ is totally split in $\Q(\zeta_d)^+$ but is not equivalent to $1\Mod{d}$, then any prime $\mathfrak{l}$ of $\Q(\zeta_d)^+$ above $l$ is inert in $\Q(\zeta_d)$.
\section{Quaternion algebras, Azumaya algebras, and Brauer groups} \label{sec:brauer}
\subsection{Quaternion algebras over fields}

In this section we recall some facts about quaternion algebras and their generalization, Azumaya algebras. A \textbf{quaternion algebra} $A$ over a field $k$  is a $4$-dimensional central simple algebra over $k$. When $k$ has characteristic different from $2$, $A$ can be described as an $k$-vector space with basis $\{1, i, j, ij\}$ and algebra structure given by $i^2=a$, $j^2=b$, and $ij=-ji$ where $a,b \in k^{*}$. One may encode this information in a \textbf{Hilbert symbol} as $\HilbertSymbol{a}{b}{k}$. We will be primarily interested in the cases where $k$ is either a number field arising as a trace field of Dehn surgeries on a knot or the function field of the canonical component of the $\SL{2}{\C}$-character variety of a hyperbolic knot. The fundamental dichotomy for quaternion algebras is that they are either division algebras or matrix algebras. Let us collect some facts about quaternion algebras that we will use later.

\begin{prop}
	\begin{enumerate} Let $k$ be a field of characteristic not equal to $2$ and $a,b,x,y \in k^{*}$.
		\item $\HilbertSymbol{a}{b}{k} \cong \HilbertSymbol{b}{a}{k}$,
		\item $\HilbertSymbol{a}{1}{k} \cong \mathrm{M}_2(k)$,
		\item $\HilbertSymbol{ax^2}{by^2}{k} \cong \HilbertSymbol{a}{b}{k}$.
	\end{enumerate}
\end{prop}

A quaternion algebra $A$ over a number field $k$ is determined up to isomorphism by the places $\mathfrak{l}$ of $F
k$ for which $A_{\mathfrak{l}}=A\otimes_k k_{\mathfrak{l}}$ is a division algebra where $k_{\mathfrak{l}}$ is the completion of $k$ with respect to $\mathfrak{l}$. The set of such $\mathfrak{l}$ is finite and of even cardinality. We say $A$ is \textbf{ramified} at these places and \textbf{split} at all the others. By abuse of notation, we occasionally write $\HilbertSymbol{a}{b}{k_{\mathfrak{l}}} = -1$ when $A$ is ramified at $\mathfrak{l}$ and $+1$ when it is split. As a justification for this notation is that, we may break up the Hilbert symbols as follows.
\[
\HilbertSymbol{a}{bc}{k_{\mathfrak{l}}} = \HilbertSymbol{a}{b}{k_{\mathfrak{l}}}\HilbertSymbol{a}{c}{k_{\mathfrak{l}}},
\]
which is essentially equivalent to quadratic reciprocity. This multiplicative notation can also be understood as equivalence in the Brauer group. In particular, $\HilbertSymbol{a}{bc}{k}\otimes_k \mathrm{M}_2(k) \cong \HilbertSymbol{a}{b}{k}\otimes_k\HilbertSymbol{a}{c}{k}$, which is equivalence in the Brauer group. We will also use this multiplicative notation for number fields where it should be interpreted as an equality of ramification sets. That is, $\HilbertSymbol{a}{bc}{k} = \HilbertSymbol{a}{b}{k}\HilbertSymbol{a}{c}{k}$ means that $\HilbertSymbol{a}{bc}{k_{\mathfrak{l}}} = \HilbertSymbol{a}{b}{k_{\mathfrak{l}}}\HilbertSymbol{a}{c}{k_{\mathfrak{l}}}$ for every prime $\mathfrak{l}$ of $k$. There is also an efficient way to compute ramification sets, which is another avatar of quadratic reciprocity.

\begin{thm}[{\cite[Theorem 2.2.6(b)]{MR}}] \label{thm:MRlocalSymbol}
	Let $A$ be a quaternion algebra over a nondyadic $\mathfrak{l}$-adic field $k_{\mathfrak{l}}$ with ring of integers $\mathcal{O}$ and maximal ideal $\mathfrak{l}$. Let $A=\HilbertSymbol{a}{b}{k_{\mathfrak{l}}}$ with $a,b \in \mathcal{O}$. If $a \notin \mathfrak{l}$ and $b \in \mathfrak{l}\setminus\mathfrak{l}^2$, then $A$ splits if and only if $a$ is a square modulo $\mathfrak{l}$.
\end{thm}
Let us also point out that quaternion algebras generate the $2$-torsion of the Brauer group of the field $k$.
\subsection{Azumaya algebras} \label{subsec:azumaya}

For a Noetherian scheme $X$, its \textbf{Brauer group} $\Br X$ is defined to be $H^2_{\mathrm{\acute{e}t}}(X, \mathbf{G}_m)$. We will have no need for the details of \'{e}tale cohomology, and in this paper one may think of $X$ as being the canonical component of a hyperbolic knot, or a smooth model thereof.  \par 

We now wish to define a generalization of quaternion algebras over fields. Let $\mathscr{O}_X$ be the structure sheaf of $X$ so that $\mathscr{O}(U)_X$ is the ring of regular functions on $U$. A coherent sheaf of $\mathscr{O}_X$ algebras is a sheaf $\mathcal{F}$ of abelian groups on $X$ such that $\mathcal{F}(U)$ is a finitely generated $\mathscr{O}_X$ algebra and the restriction maps are compatible with the algebra structure. Moreover, $\mathcal{F}$ is locally free if it has an open covering by sets $U$ such that $\mathcal{F}|_{U}$ is a free $\mathscr{O}_X|_U$ module. An \textbf{quaternion Azumaya algebra} is a nonzero $\mathscr{O}_X$-algebra that is locally free of rank $4$. \par
The connection to quaternion algebras over fields is that when $X$ is a variety over a field $k$ of characteristic not equal to $2$ and $\mathcal{A}$ is a quaternion Azumaya algebra, there is a finite open covering of $X$ by sets $U$ such that for each $U$, there is an isomorphism of $\mathscr{O}_X|_U$-modules
\[
\mathcal{A}|_U \cong \mathscr{O}|_U \oplus i\mathscr{O}|_U \oplus j\mathscr{O}|_U \oplus ij\mathscr{O}|_U,
\]
where $i^2 = f_U$, $j^2 = g_U$, and $ij=-ji$ for $f_U,g_U \in k[U]^{*}$. So quaternion Azumaya algebras locally look like quaternion algebras. Moreover, one can take the fiber at a point $\mathcal{A}(x) = \mathcal{A} \otimes_{\mathscr{O}_X} k(x)$ where $k(x)$ is the residue field to obtain a quaternion algebra over the residue field. \par
Writing $k(X)$ for the function field of $X$, there is a canonical injection $\Br X \hookrightarrow \Br k(X)$ and an exact sequence which determines its image.
\begin{thm}[{\cite[Theorem 6.8.3.]{QPoints}}] \label{thm:cohomologicalPurity}
	Let $X$ be a regular integral Noetherian scheme. Let $X^{(1)}$ be the set of codimension $1$ points of $X$. Then the sequence
	\[
		0 \longrightarrow \Br X \longrightarrow \Br k(X) \xrightarrow{\mathrm{res}} \bigoplus_{x \in X^{(1)}} H^1(k(x), \Q/\Z) 
	\]
	is exact with the caveat that one must exclude the $p$-primary part of all the groups if $X$ is of dimension $\leq 1$ and some $k(x)$ is imperfect of characteristic $p$, or if $X$ is of dimension $\geq 2$ and some $k(x)$ is of characteristic $p$.
\end{thm}
This theorem is known as absolute cohomological purity. It was conjectured by Grothendieck and proved by Gabber, though the above formulation appears in \cite{QPoints}. The last arrow is the residue homomorphism to the Galois cohomology group $H^1(k(x), \Q/\Z) = H^1(\Gal(k(x)^{sep}/k(x), \Q/\Z)$. We say that a quaternion algebra $A_{k(X)}$ defined over the function field $k(X)$ ``extends" over a point $x \in X$ if its residue is trivial at $x$. The exact sequence in Theorem \ref{thm:cohomologicalPurity} says that $A_{k(X)}$ extends to a quaternion Azumaya algebra if and only if its residue is trivial everywhere. For quaternion algebras over function fields, the residues may be calculated using a tame symbol at least when the residue field has characteristic different from $2$. Namely, let $\alpha, \beta \in k(X)$ and $x$ be a codimension $1$ point, and write
\[
	\{\alpha, \beta\} = (-1)^{\ord_x(\alpha)\ord_x(\beta)}\beta^{\ord(\alpha)}/\alpha^{\ord(\beta)} \in k(x)^{*}/k(x)^{*^2},
\]
where $\ord_x(\alpha)$ is the order of vanishing of $\alpha$ at $x$. Then this class in $k(x)^{*}/k(x)^{*^2}$ is equal to the residue of $\HilbertSymbol{\alpha}{\beta}{k(X)}$ at $x$. In particular, if this class is the trivial square class, then the algebra extends over $x$. See \cite[\S 2]{CTetalMumbai} for details. \par
An important property of Azumaya algebras of varieties over number fields is that their fibers, which are quaternion algebras over number fields, can only ramify at a finite set of places. We give the statement that appears in \cite{SkoroNotes} Theorem 2.4(2), though we point out that it holds for any Azumaya algebra, not just the quaternion ones.
\begin{thm} \label{thm:resultatClassique}
	Let $X$ be a smooth projective irreducible variety over a number field $k$, and let $\mathcal{A}$ be a quaternion Azumaya algebra on $X$. Then, for almost all places $\mathfrak{l}$, we have $\mathcal{A}(P) \cong \mathrm{M}_2(k_{\mathfrak{l}})$ for all $P \in X(k_{\mathfrak{l}})$.
\end{thm}
However, work of Harari (\cite{Harari94}) shows that if $A_{k(X)}$ has a residue, then there are infinitely many places $\mathfrak{l}$ for which there is a nontrivial fiber at some local point.
\subsection{Connection to Kleinian groups} We now explain the connection between the Azumaya algebra machinery and Kleinian groups. We refer the reader to \cite{CRS} for more details. Let $C$ denote the canonical component of a hyperbolic knot $K$. As mentioned in the introduction, there is always a quaternion algebra defined over the function field of the canonical component that specializes at a character of a hyperbolic Dehn surgery to the usual quaternion algebra of a Kleinian group. A natural question to ask is whether this quaternion algebra extends to an Azumaya algebra. The answer for hyperbolic knot complements is that it extends if and only if the Alexander polynomial satisfies condition $(\star)$ of \cite{CRS}.
\begin{thm}[\cite{CRS} Theorems 1.2., 1.4.] \label{thm:1.2OfCRS} Let $K$ be a hyperbolic knot with $\Gamma = \pi_1\left(S^3 \backslash K\right)$, and suppose that $\Delta_K$ satisfies condition $(\star)$. Then
	\begin{enumerate}
		\item $A_{k(C)}$ comes from an Azumaya algebra in $\Br \tilde{C}$ where $\tilde{C}$ denotes the normalization of the projective closure of $C$. \\
		\item Furthermore, if the canonical component is defined over $\Q$, there exists a finite set $S_K$ of rational primes such that, for any hyperbolic Dehn surgery $N$ on $K$ with trace field $k_N$, the $k_N$-quaternion algebra $A_N$ can only ramify at real places of $k_N$ and finite places lying over primes in $S_K$. 
	\end{enumerate} 
\end{thm}
For example, the authors calculate in \cite{CRS} that the figure-eight knot can have only real and dyadic ramification. There is also a partial converse in \cite{CRS} (Theorems 1.2, 1.4), namely that $A_{k(C)}$ does not extend when the knot fails condition $(\star)$. With the results of \cite{Harari94}, this implies that one can obtain ramification above infinitely many rational primes by specializing $A_{k(C)}$. However, these points \textit{a priori} need not be interesting from the point of view of geometric structures. Experimental evidence led them to conjecture (Conjecture \ref{conj:inf_bad}) that when the knot fails condition $(\star)$, there should be ramification above infinitely many rational primes, even when one restricts attention to Dehn surgery points.

\subsection{Ramification for Dehn surgery points}
We now prove that the ramification of the specializations to $(d,0)$ surgery can be expressed in terms of the following Hilbert symbol. Throughout, we write $r_d$ for the algebraic number that $r$ specializes to at $(d,0)$ surgeries to avoid confusion with the coordinate $r=2-R$.

\begin{prop} \label{prop:HilbertSymbolForKnot}
	Let $d$ be an odd positive integer that is not a power of a prime and suppose that $K$ is a $2$-bridge knot whose Alexander polynomial has only simple roots and is Azumaya negative. Let $k_d$ be the trace field of the $(d,0)$ surgery and $r_d$ as above.
	Then, there is a finite set $S$ of rational primes such that for a prime $\mathfrak{L}$ of $k_d$ lying above a prime not in $S$, the ramification of the (invariant) quaternion algebra for the $(d,0)$ surgery on $K$ agrees at $\mathfrak{L}$ with Hilbert symbol
	\[
	\HilbertSymbol{2\cos(2\pi/d)-2}{-r_d}{k_d}.
	\]
\end{prop}

\begin{rem}
	Proposition \ref{prop:HilbertSymbolForKnot} implies that if there are infinitely many primes $\mathfrak{l}$ as $d$ varies of different residue characteristics at which
	\[
	\HilbertSymbol{2\cos(2\pi/d)-2}{-r_d}{k_d},
	\]
	is ramified, then the (invariant) quaternion algebras for the $(d,0)$ surgeries have infinitely many different ramified residue characteristics as $d$ varies.
\end{rem}

\begin{proof}[Proof of Proposition \ref{prop:HilbertSymbolForKnot}]
	In general the Hilbert symbol at a representation $\rho$ is given (see Proposition \ref{CRSFunctionFieldHilbertSymbol}) by
	\[
	\HilbertSymbol{\chi_{\rho}(a)^2-4}{\chi_{\rho}([a,b])-2}{k_{\rho}}.
	\]
	Using the coordinates $Z$ and $R$ defined above, this looks like
	\begin{equation} \label{eq:breakingHilbertSymbol}
		\HilbertSymbol{Z^2-4}{2Z^2+R^2-Z^2R-4}{k(C)} = \HilbertSymbol{Z^2-4}{R-2}{k(C)}\HilbertSymbol{Z^2-4}{R+2-Z^2}{k(C)}.
	\end{equation}
	For notational convenience, call the leftmost symbol in Equation \ref{eq:breakingHilbertSymbol} $A$, the first symbol to the right of the equals sign $B$, and the rightmost symbol $C$. We claim that $C$ extends to an Azumaya algebra. Indeed, its two entries vanish only at the points determined by $R=2$, $Z=\pm 2$, and $Z^2=R+2$. For $R=2$, note that $C$ has trivial residue there. In fact specializing to $R=2$ makes the second entry $4-Z^2$, so that the algebra is split (see \cite[Corollary 2.3.3]{MR}) at all specializations outside of $Z=\pm 2$, which we treat later. Then, we find that the $B$ retains the nontrivial residue at $R=2$. Moreover, $C$ \textit{a priori} might have a nontrivial residue at $Z^2=R+2$, but neither $A$ nor $B$ does, so neither does the rightmost symbol. Finally, $B$ and $C$ have nontrivial residue at $Z=\pm 2$. However, when $Z=\pm 2$, $R-2$ is a global square in the residue field. Indeed, since $a$ and $b$ are conjugate in the fundamental group, there is a conjugation
	\[
	\begin{pmatrix}
		\alpha & \beta \\
		\gamma & \delta 
	\end{pmatrix}
	\begin{pmatrix}
		1 & 1 \\
		0 & 1
	\end{pmatrix}
	\begin{pmatrix}
		\delta & -\beta \\
		-\gamma & \alpha
	\end{pmatrix} =
	\begin{pmatrix}
		1 & 0 \\
		r & 1
	\end{pmatrix}.
	\]
	Multiplying everything out on the left-hand side yields an $\alpha^2$ in the upper right entry, so $\alpha=0$, which allows us to obtain
	\[
	\begin{pmatrix}
		0 & \beta \\
		\gamma & \delta
	\end{pmatrix}
	\begin{pmatrix}
		1 & 1 \\
		0 & 1
	\end{pmatrix}
	\begin{pmatrix}
		\delta & -\beta \\
		-\gamma & 0
	\end{pmatrix} = 
	\begin{pmatrix}
		-\beta\gamma & 0 \\
		-\gamma^2 & -\beta\gamma
	\end{pmatrix} =
	\begin{pmatrix}
		1 & 0 \\
		-\gamma^2 & 1
	\end{pmatrix}.
	\]
	Note further than $\gamma$ is in the trace field of any representation with $Z=2$. Indeed, writing $\rho(c)$ for the element effecting the above conjugation (so $\rho(cac^{-1}) = \rho(b)$), $\gamma+\delta$ and $\delta-\gamma$ are the traces of $\rho(ca)$ and $\rho(ac^{-1})$, respectively. So $r = -\gamma^2$ implying $R-2 = \gamma^2$, which is a global square in the trace field of the representation. 

This implies that the tame symbol (see Subsection \ref{subsec:azumaya}) and hence the residue is trivial. A similar argument handles $Z=-2$. Since $A$ has trivial residue at $Z=\pm 2$ and $B$ does too, then $C$ must have trivial residue there as well. Then $C$ extends to an Azumaya algebra, so by Theorem \ref{thm:resultatClassique}, there is a finite set $S$ of rational primes such that if $l \notin S$, then no specialization of the second symbol is ramified at a prime above $l$. That is for $l \notin S$ and $\mathfrak{l}$ above $l$, the ramification of $A$ and $B$ agree. \par
	Specializing to $(d,0)$ surgery sends $Z$ in as $2\cos(2\pi/d)$, so $B$ specializes to
	\[
	\HilbertSymbol{4\cos^2(2\pi/d)-4}{R_d-2}{k_{d}}=\HilbertSymbol{2\cos(2\pi/d)-2}{R_d-2}{k_{d}},
	\]
	where $R_d = 2 - r_d$. The above equality follows from the fact that when $d$ is odd, $2\cos(2\pi/d)+2$ is a global square in $\Q(\zeta_d)^{+}$, which is a subfield of $k_d$. One of its square roots is $\zeta_d^{\frac{d+1}{2}}+\zeta_d^{\frac{-(d+1)}{2}}$. Changing $R_d-2$ to $-r_d$ completes the proof.
	
\end{proof}

\section{Proof of Theorem \ref{thm:main}} \label{sec:sketch}

The basic strategy for proving Theorem \ref{thm:main} is as follows. Proposition \ref{prop:HilbertSymbolForKnot} gives an explicit description for the Hilbert symbol of $(d,0)$ surgeries on the knots. In Lemma \ref{lem:HilbertSymbolTotallyReal} we use that Hilbert symbol to determine places that the associated quaternion algebra is ramified at in terms of splitting conditions on the primes diving $r_d$. Then, in Lemmas \ref{lem:relativeNorm} and \ref{lem:absoluteNorm}, we give conditions for a prime to divide $r_d$ and for it to satisfy the appropriate splitting conditions coming from Lemma \ref{lem:HilbertSymbolTotallyReal}, respectively. Finally Lemmas \ref{lem:totallyRealTotallySplit} and \ref{lem:divisors} show how to find infinitely many such primes. The remainder of the section states these intermediate steps and explains how they add up to a proof of Theorem \ref{thm:main}. \par

Theorem \ref{thm:main} then will follow once we can prove that there are infinitely many rational primes that are residue characteristics of ramified primes of the quaternion algebra determined $\HilbertSymbol{2\cos(2\pi/d)-2}{-r_d}{k_d}$ for some $d$. The next lemma describes the ramification of this quaternion algebra in terms of the splitting of primes between $\Q(\zeta_d)^{+} = \Q(\zeta_d+\zeta_d^{-1})$ and $k_d$.

\begin{lem} \label{lem:HilbertSymbolTotallyReal}
	Let $r$ be an algebraic integer inside some fixed finite extension $k_n/\Q(\zeta_d)^{+}$ for $n$ odd. Suppose that $\mathfrak{L}$ is a prime of $k_d$ that does not lie above $2$ or $d$, divides $-r_d$ an odd number of times, and has odd inertia degree over $\Q(\zeta_d)^{+}$. Suppose further than $\mathfrak{l} = \mathfrak{L} \cap \Q(\zeta_d)^{+}$ does not split in $\Q(\zeta_d)$. Then
	\[
	\HilbertSymbol{2\cos(2\pi/d)-2}{-r_d}{k_d},
	\]
	is ramified at $\mathfrak{L}$.
\end{lem}
To prove that the quaternion algebra at the $(d,0)$ surgery is ramified at a prime $\mathfrak{L}$ of the trace field that lies above neither $2$ nor $d$, it suffices to show:
\begin{enumerate}[label={(\arabic*)},ref={(\arabic*)}]
	\item $\mathfrak{L}$ divides $-r_d$ and odd number of times, \label{stepone}
	\item $\mathfrak{L}$ has odd inertia degree over $\Q(\zeta_d)^{+}$, and \label{steptwo}
	\item $\mathfrak{l} = \mathfrak{L} \cap \Q(\zeta_d)^{+}$ does not split in $\Q(\zeta_d)$. \label{stepthree}
\end{enumerate}
Theorem \ref{thm:main} will follow if we can arrange these conditions for primes above infinitely many distinct rational primes as we vary $d$.

To handle conditions \ref{stepone} and \ref{steptwo}, we exploit a connection to the Alexander polynomial of the knot. To find primes dividing $r_d$, we compute its field norm, $N_{k_d/\Q}(r_d)$. The (absolute) field norm will be a rational integer whose prime divisors correspond to prime ideals dividing $r_d$. To compute this norm, we first find the relative field norm, $N_{k_d/\Q(\zeta_d)^+}(r_d)$, which can be expressed in terms of the Alexander polynomial of the knot. Recall from Section \ref{sec:cyclo} that $N_{k_d/\Q}(r_d) = N_{\Q(\zeta_d)^+/\Q}\left(N_{k_d/\Q(\zeta_d)^+}(r_d)\right)$ and that the Alexander polynomial of the knot $K_t$ is $\Delta_{K_t}(x)=\left(\frac{t+1}{2}\right)x^2-tx+\left(\frac{t+1}{2}\right)$.
\begin{lem} \label{lem:relativeNorm}
	Let $d$ be odd, $k_d$ the trace field of the $(d,0)$ surgery of a hyperbolic twist knot $K_t$. Then the norm of $r_d$ in the relative extension $k_d/\Q(\zeta_d)^{+}$ is $\Delta_K(\zeta_d^2)=\left(\frac{t+1}{2}\right)\zeta_d^4-t\zeta_d^2+\left(\frac{t+1}{2}\right)$ for all but finitely many $d$.
\end{lem}
The proof of the above lemma involves being able to prove that specializing the character variety at $Z=2\cos(2\pi/d)$ produces an irreducible polynomial over the field $\Q(\zeta_d)^{+}$. Calculation of the character variety can be found above in Section \ref{sec:charVars}, and irreducibility of the relevant specializations is treated in Section \ref{sec:irreducibility}. This irreducibility also somewhat justifies the notation $r_d$ as it represents an algebraic number that is well defined up to Galois conjugation. We are still left to find rational primes dividing the absolute norm, $N_{k_d/\Q}(r_d)$. This will be done in Lemma \ref{lem:divisors} once we may state precisely which rational primes ensure the desired ramification. \par 

To handle condition \ref{steptwo} about the inertia degree, we use the following lemma.

\begin{lem} \label{lem:absoluteNorm}
	Suppose that $d$ is odd and that $l$ is an odd rational prime coprime to $d$ dividing $N_{\Q(\zeta_d)^+/\Q}(\Delta_{K_t}(\zeta_d^2))$ an odd number of times. Then there is a prime $\mathfrak{L}$ of $k_d$ above $l$ such that $\mathfrak{L}$ divides $-r_d$ and odd number of times and $\mathfrak{L}$ has odd inertia degree over $\Q(\zeta_d)^+$.

\end{lem}

The next lemma applies Lemma \ref{lem:absoluteNorm} to cast condition \ref{stepthree} also in terms of the Alexander polynomial. 

\begin{lem} \label{lem:totallyRealTotallySplit}
	Suppose that $d$ is odd and that $l$ is a rational prime below a prime dividing $\Delta_{K_t}(\zeta_d^2)$. Suppose further that $l$ and $d$ are coprime to $t$ and $\frac{t+1}{2}$. Then $l$ is totally split in $\Q(\zeta_d)^+$. Hence, if $l \not\equiv 1\Mod{d}$, then all primes of $\Q(\zeta_d)^+$ above $l$ are inert in $\Q(\zeta_d)$.
\end{lem}

We work with $d$ of the form $d=p^uq^v$ for $p,q$ as in the statement of Theorem \ref{thm:main} and $u,v$ integers. Our work so far says that if a rational prime $l$ divides $N_{\Q(\zeta_d)^+/\Q}(\Delta_{K_t}(\zeta_d^2))$ an odd number of times and is not equivalent to $1 \Mod{pq}$, then there is some prime $\mathfrak{L}$ of $k_d$ above $l$ at which the quaternion algebra for the $(d,0)$ surgery is ramified. We have not yet proved that any such $l$ exist. The next lemma shows that we may find infinitely many. Its proof combines an analysis of the resultant of the Alexander polynomial with the cyclotomic polynomials and a dynamical result of Furstenberg appearing in \cite{FurstenbergNonlacunary}.

\begin{lem} \label{lem:divisors}
	Let $t \in \Z_{\geq 0}$ be odd, $\Delta(x) = \left(\frac{t+1}{2}\right)x^2-tx+\left(\frac{t+1}{2}\right) \in \Z[x]$. Let $p,q$ be distinct, rational, odd primes and suppose that $\Delta(x) \equiv 1 \Mod{pq}$. Then there are infinitely many positive integers $d$ for which $N_{\Q(\zeta_d)^+/\Q}(\Delta(\zeta_d^2))$ is divisible by a rational prime $l$ an odd number of times and $l \not\equiv 1 \mod{pq}$.
\end{lem}

Fixing an integer $t \geq 2$, we apply Lemma \ref{lem:divisors} using the polynomial $\Delta_{K_t}(x) = \left(\frac{t+1}{2}\right)x^2-tx+\left(\frac{t+1}{2}\right)$, which is the Alexander polynomial of $K_t$. Then for each of the $d$ produced by Lemma \ref{lem:divisors}, the quaternion algebra for the $(d,0)$ surgery is ramified at some prime with residue characteristic $l$ as in the statement of Lemma \ref{lem:divisors}. The proof of Theorem \ref{thm:main} will be complete one we show (see Lemma \ref{lem:multiplicativeOrder} for details) that each such $l$ can only occur for finitely many $d$.

\section{Ramification} \label{sec:ramification}
In this section, we prove Lemmas \ref{lem:HilbertSymbolTotallyReal}, \ref{lem:absoluteNorm}, \ref{lem:totallyRealTotallySplit}, and \ref{lem:divisors}.
\begin{lem} \label{lem:localSquare}
	 If $n$ is an odd positive integer and $\mathfrak{l}$ is a prime ideal of $\Q(\zeta_d)^{+}$, then $2\cos(2\pi/d) - 2$ is a nonsquare modulo $\mathfrak{l}$ if and only if $\mathfrak{l}$ is inert in $\Q(\zeta_d)$.
\end{lem}
\begin{proof}
	Note that $2\cos(2\pi/d) - 2$ is not a global square in $\Q(\zeta_d)^{+}$ as it is negative at the real embeddings. It is, however, a global square in $\Q(\zeta_d)$. To see this, note that when $d$ is odd, $\zeta_d^{(d-1)/2} \in \Q(\zeta_d)$, and
	\[
	\begin{aligned}
	\left(\zeta_d^{(d-1)/2}\left(\zeta_d-1\right)\right)^2 &= \zeta_d^{d+1} - 2 \zeta_d^{(d+1)/2+(d-1)/2} + \zeta_d^{d-1} \\
	&= 2 \cos(2\pi/d) - 2.
	\end{aligned}
	\]
	So if $\mathfrak{l}$ splits (or is ramified) in $\Q(\zeta_d)$, then $\Z[2 \cos(2\pi/d)]/\mathfrak{l}$ coincides with a quotient of $\Z[\zeta_d]$ wherein $2\cos(2\pi/d) - 2$ is a square. On the other hand if $\mathfrak{l}$ is inert, then the finite field containing the square root of the reduction of $2\cos(2\pi/d) - 2$ modulo $\mathfrak{l}$ is a proper extension of $\Z[2 \cos(2\pi/d)]/\mathfrak{l}$. That is, $2\cos(2\pi/d) - 2$ is a nonsquare modulo $\mathfrak{l}$.
\end{proof}
\begin{proof}[Proof of Lemma \ref{lem:HilbertSymbolTotallyReal}]
	This ramification is equivalent to $2\cos(2\pi/d)-2$ being a nonsquare modulo $\mathfrak{L}$ by Theorem \ref{thm:MRlocalSymbol}. By hypothesis, the inertia degree is odd, so the residue field has odd degree over the finite field $\Z[2\cos(2\pi/d)]/\mathfrak{l}$. Moreover, since $\mathfrak{l}$ is inert in $\Q(\zeta_d)$,  $2\cos(2\pi/d)-2$ is a nonsquare modulo $\mathfrak{l}$ by Lemma \ref{lem:localSquare}. Then no odd degree extension of $\Z[2\cos(2\pi/d)]/\mathfrak{l}$ can contain a square root of $2\cos(2\pi/d)-2$, so it is a nonsquare modulo $\mathfrak{L}$.
\end{proof}

\begin{proof}[Proof of Lemma \ref{lem:absoluteNorm}]
	
	Consider the norm absolute norm $N_{k_d/\Q}(-r)$. By Lemma \ref{lem:relativeNorm},
	\[
	\begin{aligned}
	N_{k_d/\Q}(-r) &= N_{\Q(\zeta_d)^+/\Q}\left(N_{k_d/\Q(\zeta_d)^+}(-r)\right) \\
	&= N_{\Q(\zeta_d)^+/\Q}(\Delta_K(\zeta_d^2)).
	\end{aligned}
	\]
	Now suppose that $l$ divides $N_{\Q(\zeta_d)^+/\Q}(\Delta_K(\zeta_d^2))$ an odd number of times. Then there is some prime $\mathfrak{L}$ of $k_d$ of odd inertia degree over $\Q$ and hence over $\Q(\zeta_d)^+$.
\end{proof}
\begin{lem} \label{lem:multiplicativeOrder}
	Let $d$ be odd and $l$ be a rational prime lying below a prime of $\Q(\zeta_d)$ dividing $\Delta_{K_t}(\zeta_d^2)$. Then $l$ has multiplicative order either $1$ or $2$ modulo $d$.
\end{lem}
\begin{proof}
	Write $\mathfrak{l}$ for a prime of $\Q(\zeta_d)$ above $l$. Note that $\Delta_{K_t}(\zeta_d^2)$ is conjugate to $\Delta_{K_t}(\zeta_d)$ over $\Q(\zeta_d)$ as long as $d$ is odd. After multiplying by roots of unity, we may assume that $\mathfrak{l}$ divides the ideal generated by $\frac{t+1}{2}(\zeta_d+\zeta_d^{-1})-t$. Observe that $l$ cannot divide $\frac{t+1}{2}$ because this implies $t\zeta_d \equiv 0 \mod{\mathfrak{l}}$, but then $l$ divides $t$ and $t+1$. Similarly, $l$ cannot divide $t$ as this implies $\left(\frac{t+1}{2} \right)(\zeta_d+\zeta_d^{-1}) \equiv 0 \mod{\mathfrak{l}}$, but $\zeta_d+\zeta_d^{-1}$ is a unit for $d$ odd, so this implies $t+1 \equiv 0 \mod{\mathfrak{l}}$. \par  
	Recall that the multiplicative order of $l$ modulo $d$ is equal to the inertia degree of $l$ in $\Q(\zeta_d)$ by Theorem \ref{thm:cyclosplitting}, so the result will follow once we show that the degree of $\Z[\zeta_d]/\mathfrak{l}$ over $\mathbf{F}_l$ is $1$ or $2$. Now, $\Delta_{K_t}(\zeta_d^2)$ has $3$ nonzero terms, so $\Delta_{K_t}(\zeta_d^2)$ dying modulo $\mathfrak{l}$ forces a linear dependence modulo $\mathfrak{l}$ of $\{1, \zeta_d, \zeta_d^2\}$, so the finite field obtained by reducing modulo $\mathfrak{l}$ must have degree only $1$ or $2$ above its prime subfield. 
\end{proof}
\begin{rem} \label{rem:finitelyManyTimes}
	It follows from Lemma \ref{lem:multiplicativeOrder} that if we fix $f$ as in the lemma, a given prime $l$ can only have a prime above it in $\Q(\zeta_d)$ dividing $\Delta_{K_t}(\zeta_d^2)$ for finitely many values of $d$ unless $l$ divides $d$ itself infinitely often as $d$ varies. In applying Lemma \ref{lem:divisors}, we will take $d$ equal to $p^uq^v$ for $p,q$ fixed distinct, odd rational primes and vary the powers $u$ and $v$. The primes $l$ produced will be $-1 \Mod{pq}$, so in particular they do not divide $d$. 
\end{rem}

\begin{proof}[Proof of Lemma \ref{lem:totallyRealTotallySplit}]
	By Lemma \ref{lem:multiplicativeOrder}, we know that $l$ has multiplicative order either $1$ or $2$. Write $\mathfrak{l}$ for a prime of $\Q(\zeta_d)^+$ above $l$ that divides $\left(\frac{t+1}{2}\right)(\zeta_d^2+\zeta_d^{-2})-t$. Note that since $d$ is odd, $\left(\frac{t+1}{2}\right)(\zeta_d^2+\zeta_d^{-2})-t$ is Galois conjugate to $\left(\frac{t+1}{2}\right)(\zeta_d+\zeta_d^{-1})-t$. Recall that $\Z[\zeta_d+\zeta_d^{-1}]$ is the ring of integers of $\Q(\zeta_d)^+$, and consider the reduction map $\Z[\zeta_d+\zeta_d^{-1}] \rightarrow \Z[\zeta_d+\zeta_d^{-1}]/\mathfrak{l}$. Then $\left(\frac{t+1}{2}\right)(\zeta_d+\zeta_d^{-1})-t$ is in the kernel of this map. That is, $(\zeta_d+\zeta_d^{-1}) \equiv \frac{2t}{t+1} \Mod{\mathfrak{l}}$, so $(\zeta_d+\zeta_d^{-1})$ lies in the prime subfield of $\Z[\zeta_d+\zeta_d^{-1}]/\mathfrak{l}$, so then $\Z[\zeta_d+\zeta_d^{-1}]/\mathfrak{l}$ is just $\mathbf{F}_l$. \par
	The assertion about inertia follows from recalling (see Theorem \ref{thm:cyclosplitting}) that the totally split primes of $\Q(\zeta_d)/\Q$ are exactly those that are $1 \Mod{d}$.
\end{proof}

Now we must establish some control on the primes dividing $\Delta_{K_t}(\zeta_d^2)$. For notational convenience, we write $f(x)$ in place of $\Delta_{K_t}(x)$ in the next lemma.
\begin{lem} \label{lem:residueClassSplitResultant}
	Let $p,q$ be distinct odd primes, $t$ as in the statement of Theorem \ref{thm:main}, $n$ an odd positive integer, and $f(x) = \left(\frac{t+1}{2}\right)x^2-tx+\left(\frac{t+1}{2}\right)$. Further, let
	\[\begin{aligned}
		w &= \dfrac{1}{\sqrt{2(t+1)}}\left(\sqrt{2t+1}+i\right) \\
		y &= \sqrt{\dfrac{t+1}{2}}w.
	\end{aligned}\]
	Then
	\begin{enumerate}
		\item $w$ is a square root of a root of $f$. \\
		\item $i\left(\overline{y}^n-y^n\right) \in \Z$, and
		\[
		i\left(\overline{y}^n-y^n\right) \equiv
		\begin{cases}
			1\Mod{pq} & n \equiv 1\Mod{4} \\
			-1\Mod{pq} & n \equiv 3\Mod{4}.
		\end{cases}
		\]
	\end{enumerate}
\end{lem}
\begin{proof}
	The assertion that $w$ is a square root of a root of $f$ may be checked by direct calculation or with software. We note as well here that $\abs{w} = 1$ as $t$ was assumed real and positive in the statement of Theorem \ref{thm:main}.
	For the second claim, we first show that $i\left(\overline{y}^n-y^n\right)$ is an integer. Consider the resultant
	\begin{equation*}
		\begin{aligned} 
			\res_x\left(x^n-1, f(x^2)\right) &= \res_x\left(x^n-1, \frac{t+1}{2}(x-w)(x+\overline{w})(x+w)(x-\overline{w})\right) \\
			&=\res_x\left(x^n-1,\sqrt{\frac{t+1}{2}}(x-w)(x+\overline{w})\right)\res_x\left(x^n-1,\sqrt{\frac{t+1}{2}}(x+w)(x-\overline{w})\right).
		\end{aligned}
	\end{equation*}
	We first observe that the two factors above are equal in absolute value. Indeed, recalling that $n$ is assumed odd, we compute
	\[
	\begin{aligned}
		\res_x\left(x^n-1,\sqrt{\frac{t+1}{2}}(x-w)(x+\overline{w})\right) &= \left(\sqrt{\frac{t+1}{2}}\right)^n\left(w^n-1\right)\left(\left(-\overline{w}\right)^n-1\right) \\
		&= \left(\sqrt{\frac{t+1}{2}}\right)^n\left(-(w\overline{w})^{n}+\overline{w}^n-w^n+1\right) \\
		&= \left(\overline{y}^n-y^n\right).
	\end{aligned}
	\]
	Similarly the other factor of $\res_x\left(x^n-1, f(x^2)\right)$ is the complex conjugate of the one just considered, so the two factors are equal in absolute value. They also lie on the imaginary axis, so multiplying by $i$ moves them to the real axis. Then, that $i\left(\overline{y}^n-y^n\right)$ is an integer will follow once we show that $\res_x\left(x^n-1, f(x^2)\right)$ is a square integer. Indeed, writing $\Phi_d$ for the $d$th cyclotomic polynomial, we have
	\[
	\begin{aligned}
		\res_x\left(x^n-1, f(x^2)\right) &= \prod_{d \mid n}\res_x(\Phi_d(x), f(x^2)) \\
		&= \prod_{d \mid n} N_{\Q(\zeta_n)/\Q}(f(\zeta_d^2)).
	\end{aligned}
	\]
	Since $f$ is an integer polynomial, $f(\zeta_d^2)$ is an algebraic integer, so its norm is in $\Z$. Furthermore, $f(\zeta_d^2)$ is in $\Q(\zeta_d)^+$, which has $\Q(\zeta_d)$ as a quadratic extension, so that its norm in $\Q(\zeta_d)/\Q$ must be a square. So then, $i\left(\overline{y}^n-y^n\right)$ is a real square root of a square integer. That is, $i\left(\overline{y}^n-y^n\right) \in \Z$. \par
	To show the claim about the residue modulo $pq$, the strategy is to first show that $i\left(\overline{y}-y\right) \equiv 1\Mod{pq}$ and $i\left(\overline{y}^3-y^3\right) \equiv -1\Mod{pq}$, then show that $y^5 \equiv y \Mod{pq}$ and $\overline{y}^5 \equiv \overline{y} \Mod{pq}$. We compute $i\left(\overline{y}-y\right) = 1$ and $i(\overline{y}^3-y^3) = \left(3m+1\right)/2$. However, since $pq \mid \dfrac{t+1}{2}$ and $t \equiv -1\Mod{pq}$, we have that $(3t+1)/2 = \dfrac{t+1}{2} + t \equiv -1\Mod{pq}$. Now we turn to showing that $y^5 \equiv y \Mod{pq}$. We remark that this reduction is slightly more delicate since $y$ is not a rational (or in fact an algebraic) integer, so the reduction map is really $\mathcal{O}[1/2] \rightarrow \mathcal{O}[1/2]/(pq)$, where $\mathcal{O}$ is the ring of integers of the extension $\Q(\sqrt{2t+1}, i)$. However, we may still compute
	\[
	y^5 = \dfrac{1}{8}\left(\left(t^2-4t-1\right)\sqrt{2t+1}+i\left(5t^2-1\right)\right),
	\]
	and observe that $t^2-4t-1 \equiv 5t^2-1 \equiv 4 \Mod{pq}$, so since
	\[
	y = \dfrac{1}{2}\left(\sqrt{2t+1}+i\right),
	\]
	we deduce that $y^5 \equiv y \Mod{pq}$. The analogous computation shows that $\overline{y}^5 \equiv \overline{y} \Mod{pq}$.
\end{proof}
\begin{lem} \label{lem:productEqualities}
	Let $t,f,w$, and $y$ be as in Lemma \ref{lem:residueClassSplitResultant}. If $n \equiv 1\Mod{4}$, then
	\[
	\prod_{d \mid n} N_{\Q(\zeta_d)^+/\Q}(f(\zeta_d^2)) = i(\overline{y}^n-y^n) = 2\Im(y^n) = 2\left(\sqrt{\frac{t+1}{2}}\right)^n\Im(w^n).
	\]
\end{lem}
\begin{proof}
	As in the proof of Lemma \ref{lem:residueClassSplitResultant}, $\left(i(\overline{y}^n-y^n)\right)^2$ is equal in absolute value to $\res_x(x^n-1,f(x))$. On the other hand, $\res_x(x^n-1,f(x))$ is also equal to $\prod\limits_{d \mid n} N_{\Q(\zeta_n)/\Q}(f(\zeta_d^2))$. Note, however, that each term in this product is in $\Q(\zeta_d)^{+}$ since $f(x)$ is a reciprocal polynomial. So
	\[
	\begin{aligned}
		\res_x(x^n-1,f(x)) &= \prod_{d \mid n} N_{\Q(\zeta_n)/\Q}(f(\zeta_d^2)) \\
		&= \left(\prod_{d \mid n} N_{\Q(\zeta_d)^+/\Q}(f(\zeta_d^2))\right)^2.
	\end{aligned} 
	\]  
	Note that, after possibly multiplying by a root of unity, $f(\zeta_d^2) = \left(\frac{t+1}{2} \right)(\zeta_d^2+\zeta_d^{-2}) - t$. Let $p,q$ be as in Lemma \ref{lem:residueClassSplitResultant}. The hypotheses that $pq \mid \frac{t+1}{2}$ and $t \equiv -1 \mod{pq}$ imply $f(\zeta_d^2) \equiv 1 \Mod{pq}$, so its norm is as well. Hence $\prod\limits_{d \mid n} N_{\Q(\zeta_d)^+/\Q}(f(\zeta_d^2))$ and $i(\overline{y}^n-y^n)$ are equal in absolute value and agree modulo $pq$  by Lemma \ref{lem:residueClassSplitResultant}, so they are the same integer. The other equalities are direct calculations.
\end{proof}
We now want to show that we can change the sign of $\Im(w^n)$ infinitely often as we vary $n$ through powers of $p$ and $q$. To do this, we use a result of Furstenberg. Before stating it, we recall that a multiplicative semigroup of the integers is called \textbf{lacunary} if it consists of powers of a single integer and \textbf{non-lacunary} otherwise.
\begin{thm}[\cite{FurstenbergNonlacunary} Theorem IV.1] \label{thm:nonlacunary}
	If $\Sigma$ is a non-lacunary semigroup of integers and $\eta$ is irrational, then $\Sigma \eta$ is dense modulo $1$.
\end{thm}
\begin{lem} \label{lem:quadraticFurstenberg}
	Let $f(x) = ax^2+bx+a \in \Z[x]$ be a reciprocal polynomial of degree $2$ with $b/a < 2$ and $p$ and $q$ distinct positive integers. Let $w$ be a square root of a root of $f$. Then for infinitely many pairs of positive integers $(u, v)$, we have $\Im(w^n) < 0$ where $n = p^{u}q^{v}$. Moreover, each $u$ and $v$ may be taken to be even.
\end{lem}
\begin{proof}
	Since $w$ is on the unit circle, we may write it as $w=\exp(2\pi i \eta)$ so that $w^n = \exp(2n \pi i \eta)$. If we let $\Sigma = \{p^{u}q^{v}\}$, then Furstenberg's Theorem \ref{thm:nonlacunary} implies that $\Sigma \eta$ is dense modulo $1$ so that for infinitely many $n \in \Sigma$, $w^n$ is in the lower-half plane. If we desire each $u$ (resp. $v$) produced to be even, we may replace $p$ (resp. $q$) with $p^2$ (resp. $q^2$) in the definition of $\Sigma$.
\end{proof}
The above lemma holds if the less than sign is replaced with a greater than sign.

\begin{proof}[Proof of Lemma \ref{lem:divisors}]
	Let $n = p^uq^v$ for $p,q$ distinct odd primes and $u,v$ positive integers. If $p$ (resp. $q$) is equivalent to $3\Mod{4}$, then suppose that $u$ (resp. $v$) is even. Note that $\prod\limits_{d \mid n} N_{\Q(\zeta_d)^+/\Q}(f(\zeta_d^2))=2\Im(y^n)$ is always $1\Mod{pq}$ when $n\equiv 1\Mod{4}$ by Lemma \ref{lem:residueClassSplitResultant}. Moreover, by Lemma \ref{lem:productEqualities}, it is equal to $2\left(\sqrt{\frac{t+1}{2}}\right)^n\Im(w^n)$, so when $\Im(w^n)$ is negative, the absolute value of $\prod\limits_{d \mid n} N_{\Q(\zeta_d)^+/\Q}(f(\zeta_d^2))$ must be equivalent to $-1 \Mod{pq}$. Then one of the $N_{\Q(\zeta_d)^+/\Q}(f(\zeta_d^2))$ must be not equivalent to $1\Mod{pq}$, so it must have a rational prime divisor to an odd power (note that all prime divisors must have multiplicative order either $1$ or $2$ modulo $d=p^{u'}q^{v'}$ by Lemma \ref{lem:multiplicativeOrder}, so if all the powers were even, $\abs{N_{\Q(\zeta_d)^+/\Q}(f(\zeta_d^2))}$ would be equivalent to $1 \Mod{p^{u'}q^{v'}}$) that is not $1 \Mod{pq}$, so by Lemma \ref{lem:totallyRealTotallySplit}, it lies below a prime $\mathfrak{l}$ of $\Q(\zeta_d)^+$ that is inert in $\Q(\zeta_d)$. One may repeatedly apply Lemma \ref{lem:quadraticFurstenberg} to change the sign of $\Im(w^n)$ back and forth to produce infinitely many $d$ with $ N_{\Q(\zeta_d)^+/\Q}(f(\zeta_d^2))$ not equivalent to $1 \Mod{pq}$.
\end{proof}
\begin{rem}
	The infinitely many $d$ produced by Lemma \ref{lem:divisors} will provide $(d,0)$ surgeries for which quaternion algebra has infinitely many distinct residue characteristics.
\end{rem}
\section{Irreducibility} \label{sec:irreducibility}
The goal of this section is to prove Lemma \ref{lem:relativeNorm}. The main ingredient will be the following proposition.
\begin{prop} \label{prop:irreducibility}
	Let $l$ be odd and $f_l(R,Z)$ the polynomial defining the canonical component of the character variety for the knot $T_l$. Then $f_l(R, 2\cos(2\pi/n))$ is irreducible as an element of $\Q^{ab}[R]$ for all but finitely many $n$.
\end{prop}
We will ultimately apply the following theorem of Dvornicich and Zannier.
\begin{thm}[{\cite[Corollary 1(a)]{DZannier}}] \label{thm:DZ}
	Let $k$ be a number field and $k^c$ the field obtained by adjoining all roots of unity to $k$. If $g \in k^c[R,Z]$ and $g(R, Z^m)$ is irreducible in $k^c[R,Z]$ for all positive integers $m \leq deg_R{g}$, then $g(R, \zeta)$ is irreducible in $k^c[R]$ for all but finitely many roots of unity $\zeta$.
\end{thm}
To apply this theorem, we actually consider the polynomials $g_t(R,Z) = Z^{\deg_Z{f_t}}f_t(R,Z+Z^{-1})$, so that specializing $g_t$ at $Z=\zeta_n$ yields $\zeta_n^{\deg_Z{f_t}}f_t(R, 2\cos(2\pi/n))$. In particular, $g_t(R, \zeta_n)$ is irreducible over $\Q(\zeta_n)$ if and only if $f_t(R, 2\cos(2\pi/n))$ is. We record the analogous formula for $g_t(R,Z) = Z^2f_t(R,Z+Z^{-1})$ here.
\begin{lem} \label{lem:g_lFormula}
	Let $t$ be an odd positive integer.
	\[
		g_t(R,Z) =
			\begin{cases}
				Z^2R^t - (Z^4+Z^2+1) + \sum\limits_{i=1}^{t-1}\left(a_iZ^2 - b_i \left(Z^4+2Z^2+1\right)\right)R^i  & t\equiv 1\Mod{4} \\
				Z^2R^t - Z^2 + \sum\limits_{i=1}^{t-1}\left(a_iZ^2 - b_i \left(Z^4+2Z^2+1\right)\right)R^i  & t\equiv 3\Mod{4},
			\end{cases}
	\]
	where $a_i, b_i \in \Z$ and $a_{t-1}=b_{t-1}=1$.
\end{lem}
Our strategy will be first to show that $f_l(R,Z^m)$ is irreducible for all positive integers $m$ and then to show that the irreducibility of $f_l(R,Z^m)$ implies that of $g_l(R,Z^m)$. We begin by proving that---over $\Q$---changing $Z$ to $Z+Z^{-1}$ and clearing denominators does not affect irreducibility.
\begin{lem} \label{lem:poweredUpCharacterVarietyIrreducibility}
	Let $f_l(R,Z)$ be as above. Then for all positive integers $m$, $f_l(R,Z^m)$ is absolutely irreducible.
\end{lem}
We need a version of Capelli's theorem due to Kneser. This formulation appears in \cite{SchinzelPolynomials}.
\begin{thm}[{\cite[Theorem 19]{SchinzelPolynomials}}] \label{thm:Capelli}
	Let $k$ be a field and $n$ an integer $\geq 2$. Let $a \in k$. The binomial $Z^n-a$ is reducible over $k$ if and only if either $a=b^p$ for some prime divisor $p$ of $n$ or $4 \mid n$ and $a=-4b^4$.
\end{thm}
\begin{proof}[Proof of Lemma \ref{lem:poweredUpCharacterVarietyIrreducibility}]
	Let $f_l(R,Z)$ as an element of $\C(R)[Z]$. Since polynomials in $R$ are units in this ring, we may clear the denominators to obtain the monic polynomial
	\[
		F_l(R,Z^m) =
			\begin{cases}
				Z^{2m} -  \dfrac{R^l + 1 + \sum\limits_{i=1}^{l-1}a_iR^i}{1+\sum\limits_{i=1}^{l-1}b_iR^i} & l\equiv 1\Mod{4} \\
				Z^{2m} - \dfrac{R^l - 1 + \sum\limits_{i=1}^{l-1}a_iR^i}{\sum\limits_{i=1}^{l-1}b_iR^i}  & l\equiv 3\Mod{4},
			\end{cases}
	\]
	which is irreducible if and only if $f_l(R,Z)$ is. Let $a$ denote the constant term and write $a=\dfrac{\alpha(R)}{\beta(R)}$ where $\alpha$ and $\beta$ are coprime in $\C[R]$. Suppose that $a=b^p$ for some prime divisor $p$ of $2m$. Then, there are polynomials $A_l, B_l \in \C[R]$ such that $A_l^p = \alpha_l$ and $B_l^p = \beta_l$. So then for some nonnegative integers $l_1$, $l_2$, we have $\deg(\alpha_l) = pl_1$ and $\deg(\beta_l) = pl_2$. However, $\alpha_l$ is of degree $l-j$ and $\beta_l$ is of degree $l-j-1$ where $j$ is the degree of the original common factors between the numerator and denominator of $a$. In particular, $\deg(\alpha_l)$ and $\deg(\beta_l)$ are comprime integers unless $j=l-1$. In this case, however, $\deg(\alpha_l) = 1$, so it can't be a $p$th power. Finally to apply Theorem \ref{thm:Capelli}, we must check that $a \neq -4b^4$ when $4 \mid 2m$. However, the analogous degree considerations show that $-a/4$ cannot be a fourth power. Then we apply Gauss's lemma to conclude that $f_l(R,Z^m)$ is irreducible as an element of $\C[R,Z]$.
\end{proof}
\begin{lem} \label{lem:changeOfVarIrreducibility}
	Let $g_l(R,Z)$ be as above. Then for all positive integers $m$, $g(R,Z^m)$ is irreducible in $\Q[R,Z]$.
\end{lem}

\begin{proof}[Proof of Lemma \ref{lem:changeOfVarIrreducibility}]
	Fix a positive integer $m$ and suppose that $g_l(R,Z^m) = G_1(R,Z)G_2(R,Z)$ where $G_i(R,Z) \in \Q[R,Z]$ of positive $R$-degree. Note that there cannot be a factorization involving a term of $R$-degree $0$ since---in the case that $l \equiv 1 \Mod{4}$---the $R$-leading coefficient, $Z^{2m}$ is coprime with the $R$-degree zero term, $-Z^{4m}-Z^{2m}-1$. When $l\equiv 3\Mod{4}$, the gcd of the leading and constant terms is $Z^{2m}$, but reducing modulo $Z$ produces $-\sum_{i=1}^{l-1}b_iR^i$, which is nonzero by Lemma \ref{lem:g_lFormula}. Since the $R$-leading coefficient of $g_l(R,Z^m)$ is $Z^{2m}$, specializing $Z=1$ produces a nontrivial factorization in $\Q[R]$. However, $g_l(R,1^m) = f_l(R,1^m+1^{-m}) = f_l(R,2)$, which is irreducible by \cite{HosteShanahanTwistKnots}.
\end{proof}
Now we prove
\begin{lem} \label{lem:gAbsolutelyIrreducible}
	Let $g_l$ be as above. Then $g_l(R,Z^m)$ is absolutely irreducible for all positive integers $m$.
\end{lem}
We apply the following result of Bertone, Ch\`eze, and Galligo.
\begin{thm}[{\cite[Proposition 3]{BCGfactorization}}] \label{thm:BCG}
	Let $k$ be a field and $g(R,Z) \in k[R,Z]$ be an irreducible polynomial. Let $\{(i_1, j_1), \dots, (i_l, j_l)\} \subset \Z^2$ be the vertex set of its Newton 
	polygon. If $\mathrm{gcd}(i_1, j_1, \dots,i_l, j_l) = 1$, then $f(R,Z)$ is irreducible over $\overline{k}$.
\end{thm}
\begin{proof}[Proof of Lemma \ref{lem:gAbsolutelyIrreducible}]
	\begin{figure}[h]
		\includegraphics[height=0.25\textheight]{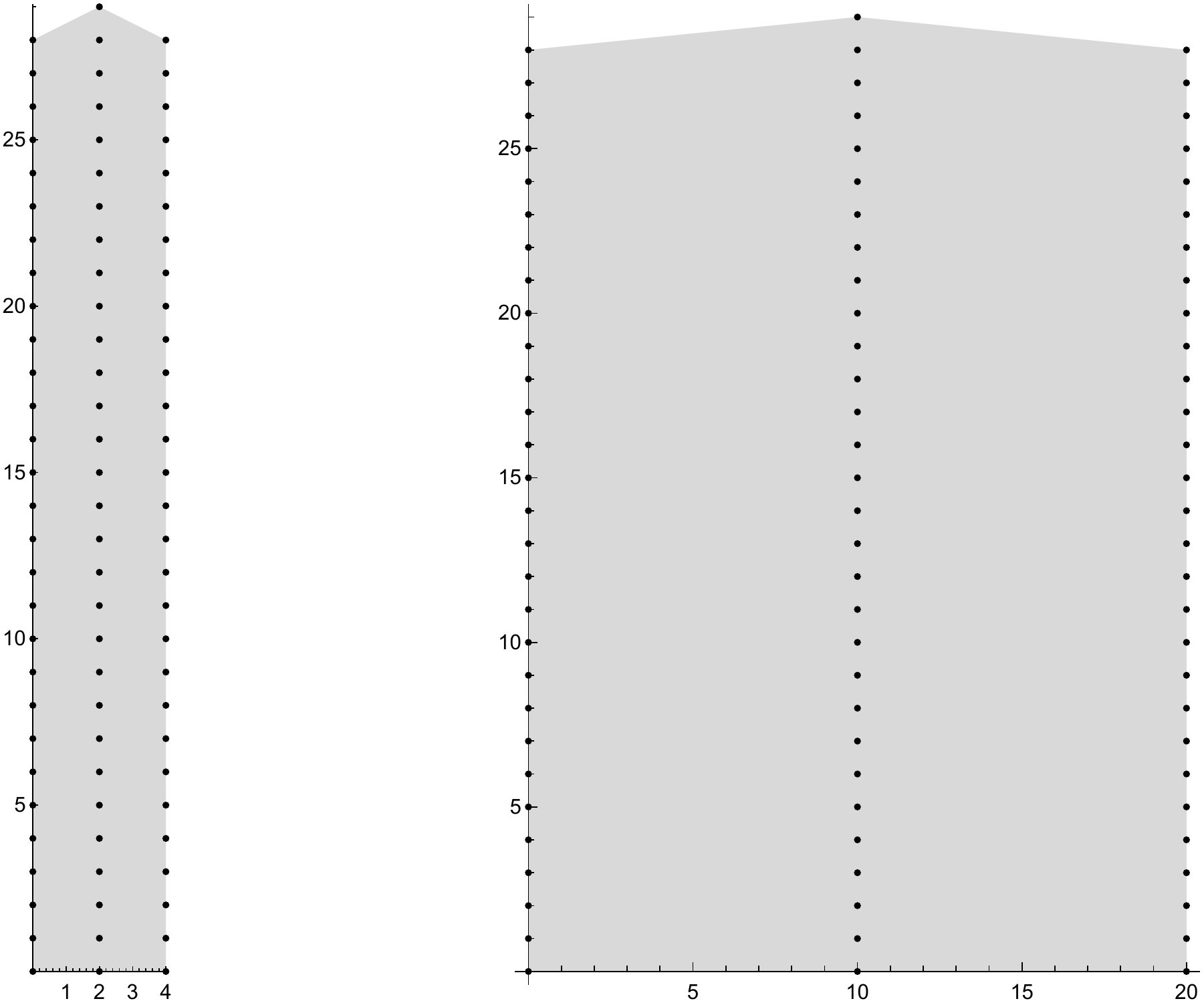}
		\caption{The Newton polygons for $g_{29}(R,Z)$ and $g_{29}(R^5,Z)$.}
		\label{fig:NewtonPolygon}
	\end{figure}
	Note that the coefficients of $R^{t-1}$, $Z^2R^t$, and $Z^4R^{t-1}$ in $g_t(R,Z)$ are all nonzero. Moreover, the coefficient of $R^t$ and $Z^4R^t$ is zero, and there are no terms with $Z$ to an odd power or a power greater than $4$. All this implies that the top of the Newton polygon (with $Z$ and $R$ the horizontal and vertical directions respectively) looks like $(0, t-1)$, $(2, t)$, $(4, t-1)$. See Figure \ref{fig:NewtonPolygon} for the Newton polygon of $g_{13}(R,Z)$. These lattice points alone are sufficient to apply Theorem \ref{thm:BCG} together with the rational irreducibility furnished by Lemma \ref{lem:changeOfVarIrreducibility} to conclude that $g_t(R,Z)$ is absolutely irreducible. To apply theorem \ref{thm:DZ}, we must further show that $g_t(R,Z^m)$ is irreducible for all $m \leq \deg_R(g_t) = t$. Indeed, $g_t(R,Z^m)$ is irreducible over $\Q$ for all positive integers $m$. Increasing the power on $Z$ has the effect of horizontally stretching the Newton polygon so that its top consists of the lattice points $(0, t-1)$, $(2m, t)$, $(4m, t-1)$. The coordinates of these points still have $t$ and $t-1$, which are coprime, so Theorem \ref{thm:BCG} still applies, so we find that $g_t(R, Z^m)$ is absolutely irreducible for all positive integers $m$. Then Theorem \ref{thm:DZ} applies so that $g_t(R, \zeta_n)$ is irreducible in $\Q^{ab}[R]$, and hence $f_t(R, 2\cos(2\pi/n))$ is as well.
\end{proof}
\begin{proof}[Proof of Lemma \ref{lem:relativeNorm}]
	We may find a monic polynomial that $r_d$ satisfies by specializing the character variety defining polynomial $f_t(R,Z)$ to $f_t(2-r,2\cos(2\pi/d))$. The resulting polynomial in $\Q(\zeta_d^+)[r]$ is irreducible by Proposition \ref{prop:irreducibility}, so $f_t(2-r,2\cos(2\pi/d))$ is the minimal polynomial for $r_d$. Moreover, its constant term coincides with that of $f_t(2,2\cos(2\pi/d))$, which is $-\Delta_{K_t}(\zeta_d)$ by Lemma \ref{lem:constantTerm}. The norm of $r_d$ is then $\Delta_{K_t}(\zeta_d)$ as the norm is the negative of the constant term of the minimal polynomial when the degree of the minimal polynomial is odd.
\end{proof}
\section{Example: \texorpdfstring{$T_{29}$}{T29}}
In this section we wish to show what ramification can be predicted for surgeries on the knot $T_{29}$ by unpacking the proof of Theorem \ref{thm:main}. We remark that for all but the smallest surgeries, it seems totally intractable to use any sort of software to na\"ively compute the ramification. One can use the character variety to get exact minimal polynomials for the trace fields and entries for the quaternion algebra in terms of this minimal polynomial, but computing reduction modulo primes (i.e. trying to apply Theorem \ref{thm:MRlocalSymbol}) involves computing a maximal order in the field, which usually involves factoring a large discriminant. For example, the discriminant of a minimal polynomial for the $(11,0)$ surgery on $T_{29}$ has $359$ decimal digits, which exceeds the record for the largest factored integer not of a special form by over $100$ decimal digits. We will not dwell on the exact nature of the computational complexities involved, but instead use some of the theoretical results in the paper to prove ramification. \par
Let us first mention that the canonical component is of the following form.
\[
\begin{aligned}
f_{29}(R,Z) &=-R^{28} Z^{2} + R^{29} + R^{27} Z^{2} + R^{28} + 26 R^{26} Z^{2} - 28 R^{27} - 25 R^{25} Z^{2} - 27 R^{26} - 301 R^{24} Z^{2} \\ 
&\qquad+ 351 R^{25} + 277 R^{23} Z^{2} + 325 R^{24} + 2046 R^{22} Z^{2} - 2600 R^{23} - 1792 R^{21} Z^{2} - 2300 R^{22} \\ 
&\qquad- 9066 R^{20} Z^{2} + 12650 R^{21} + 7506 R^{19} Z^{2} + 10626 R^{20} + 27492 R^{18} Z^{2} - 42504 R^{19} \\ 
&\qquad- 21335 R^{17} Z^{2} - 33649 R^{18} - 58277 R^{16} Z^{2} + 100947 R^{17} + 41941 R^{15} Z^{2} + 74613 R^{16} \\ 
&\qquad+ 86662 R^{14} Z^{2} - 170544 R^{15} - 57044 R^{13} Z^{2} - 116280 R^{14} - 89402 R^{12} Z^{2} + 203490 R^{13} \\ 
&\qquad+ 52834 R^{11} Z^{2} + 125970 R^{12} + 62292 R^{10} Z^{2} - 167960 R^{11} - 32206 R^{9} Z^{2} - 92378 R^{10} \\ 
&\qquad- 27966 R^{8} Z^{2} + 92378 R^{9} + 12174 R^{7} Z^{2} + 43758 R^{8} + 7476 R^{6} Z^{2} - 31824 R^{7} - 2576 R^{5} Z^{2} \\ 
&\qquad- 12376 R^{6} - 1036 R^{4} Z^{2} + 6188 R^{5} + 252 R^{3} Z^{2} + 1820 R^{4} + 56 R^{2} Z^{2} - 560 R^{3} - 7 R Z^{2} - 105 R^{2} \\ 
&\qquad- Z^{2} + 15 R + 1.
\end{aligned}
\]
Let us first consider the $(11,0)$ surgery on $T_{29}$, so our quaternion algebra is of the form
\[
\HilbertSymbol{2\cos(2\pi/11)-2}{-r_{11}}{k_{11}}.
\]
Note that this is not actually the quaternion algebra associated to the $(11,0)$ surgery, but instead the specialization of the Azumaya negative part (see Proposition \ref{prop:HilbertSymbolForKnot} for details). However, all ramification shown in this example will actually be associated to the quaternion algebra for the Kleinian group. This basically amounts to checking that the Azumaya positive part doesn't have any of the same ramification. We also remark that $k_{11}$ is a degree $145$ extension of the $\Q$. Without knowing something about the ring of integers (e.g. the conductor of $\Z[r_d]$ in it), it is difficult to explicitly establish the splitting of primes in this extension. However, applying Lemma \ref{lem:HilbertSymbolTotallyReal}, one may deduce the existence of a ramified prime. In particular, we first need to prove that some prime $\mathfrak{L}$ of $k_{11}$ divides $-r_{11}$ an odd number of times. This is a problem amenable to software. Indeed, $N_{k_{11}/\Q}(-r_{11}) = 43\cdot131\cdot1033$. Since these all appear to the first power, there is no worry of a prime dividing $-r_{11}$ multiple times. Now, we check the other criterion for these rational primes, namely that any prime of $\Q(\zeta_{11})^{+}$ above them does not split in $\Q(\zeta_{11})$. Since the tower of extensions $\Q(\zeta_d)/\Q(\zeta_d)^{+}/\Q$ is Galois, each prime above a particular rational prime has the same splitting behavior. We may check that each of $43$, $131$, and $1033$ satisfy the conditions of Lemma \ref{lem:HilbertSymbolTotallyReal}, so there is a prime in $k_{11}$ above each of these rational primes such that the quaternion algebra for $T_{29}(11,0)$ is ramified at that prime. This procedure can be implemented on a computer; after doing so, we find the residue characteristics listed in Table \ref{tab:residueChars} of ramified primes for the $(d,0)$ surgery.
\begin{table} 
\centering
	\caption{Ramified residue characteristics for $(d,0)$ surgery on $T_{29}$.}
	\label{tab:residueChars}
	\begin{tabular}{c|c c@{\hspace{0.25in}}c|c}
		$d$ & primes & & $d$ & primes  \\
		\cline{1-2} \cline{4-5} 
		5 &  $\varnothing$               & & 53 &  42611, 60101 \\
		7 & 13                           & & 55 &  $\varnothing$ \\     
		9 & 431                          & & 57 &  12539, 56706539232099509                \\       
		11 & 43,131,1033                 & & 59 &  1228786844647                \\         
		13 & 1117, 1481                  & & 61 &  35711295669608681, 41553136798440921281 \\
		15 & 149, 179                    & & 63 &  $\varnothing$          \\      
		17 & 67, 101, 509, 4657          & & 65 &  4679, 656305837760821827656999          \\
		19 & 37                          & & 67 &  401       \\             
		21 & $\varnothing$               & & 69 & $\varnothing$                   \\        
		23 & 10938592571969              & & 71 & 283               \\          
		25 & 90636599549                 & & 73 & 5503792674161   \\
		27 & $\varnothing$               & & 75 & 2699, 15299                   \\        
		29 & 292319                      & & 77 & 5196259971209 \\
		31 & $\varnothing$               & & 79 & 157, 80263                         \\
		33 & 659, 24800291               & & 81 & 314974336585075469        \\ 
		35 & 25409                       & & 83 & 74201, 33552749, 27639164173, 19501822788835693 \\
		37 & 73, 294149, 531516948137827 & & 85 & 123419, 4093091532209, 16729850810909 \\
		39 & 35883041                    & & 87 & 7386583213044449, 65955561202472999   \\
		41 & 4271162617                  & & 89 & 1601                          \\
		43 & 3697, 107069                & & 91 & 42223, 122828797084811                        \\
		45 & 89                          & & 93 & 929, 46197763017488779460706369300779   \\
		47 & $\varnothing$               & & 95 & 569, 145349, 153862768739      \\
		49 & 97                          & & 97 & 79151, 149328007, 3899539084760806682641718399966621 \\
		51 & 1055801, 823976217011       & & 99 & 2970791, 3683326481, 9934540457447231.
	\end{tabular}
\end{table} \par
The above discussion is meant to illustrate that results in the paper allow relatively easy computation of the residue characteristics of ramified primes for $(d,0)$ surgeries. Of course, Theorem \ref{thm:main} indicates that there are in fact infinitely many distinct such residue characteristics. We now show how one can unpack the methods of the proof to find a subsequence of surgery coefficients that provides the infinitely many distinct residue characteristics. We can summarize previous work in the paper by saying that we ultimately want to find $d$ such that
\[
\left| N_{\Q(\zeta_d)^+/\Q}(\Delta_{T_{29}}(\zeta_d^2)) \right| \not\equiv 1 \Mod{15}.
\]
The detected surgery coefficients for this example are of the form $3^u5^v$. To simplify the explanation for this example, let us define a function
\[
\omega(n) = \prod_{d \mid n} N_{\Q(\zeta_d)^+/\Q}\left(\Delta_{T_{29}}(\zeta_d^2)\right).
\]
Lemma \ref{lem:residueClassSplitResultant} implies that $\omega(n) \equiv 1 \Mod{15}$. The idea is then that if $n_0 \mid n_1$ and $\sigma(n_0)$ and $\sigma(n_1)$ have different signs, then $\abs{\sigma(n_1)/\sigma(n_0)} \not\equiv 1 \Mod{pq}$, so there is a divisor $d_1$ of $n_1$ that does not divide $n_0$ such that 
\[
\left| N_{\Q(\zeta_{d_1})^+/\Q}(\Delta_{T_{29}}(\zeta_{d_1^2})) \right| \not\equiv 1 \Mod{15}.
\]
Then the problem is just reduced to showing that for any $n_0$, we can find a $n_1$ as above. For the knot $T_{29}$, this amounts to asking whether---given powers $u$ and $v$---one can find powers $u' > u$ and $v' > v$ such that $\omega(3^{u'}5^{v'})$ has a different sign from that of $\omega(3^u5^v)$. We prove that one can always do this by using Furstenberg's theorem on nonlacunary semigroups (see Theorem \ref{thm:nonlacunary}). However, for the present example, we just compute a few cases to show how one might go about constructing infinitely many distinct residue characteristics. \par
For technical reasons, we actually need even powers of $3$. So our first $n_0$ to try is $3^25 = 45$. $\omega(45)$ is negative, so $\abs{\omega(45)}$ has a divisor that works. In fact, consulting the table above, there are three divisors $d$ of $45$ so that the $(d,0)$ surgery has finite ramification, namely $9$, $15$, and $45$. Call one of them $d_0$. The smallest $n_1$ such that $\omega(n_1) > 0$ is $n_1 = 3^25^5 = 28125$. So there must be some divisor $d_1$ of $28125$ that doesn't divide $45$ (and hence $d_0$) such that the $(d_1,0)$ surgery has finite ramification, for example $75$. Next, if we set $n_2 = 3^25^6$, then $\omega(n_2) < 0$, so we may find $d_2$ in the same way. The numbers involved quickly become intractably large even for software, but results in the paper guarantee that proceeding in this manner produces an infinite sequence of $(d_i,0)$ surgeries with finite ramification. Moreover, because of Lemma \ref{lem:multiplicativeOrder}, we can't find the same residue characteristic infinitely often, so this process will in fact produce infinitely many distinct rational primes such that for each one, there is an integer $d$ and a prime of the trace field of the $(d,0)$ surgery above that prime that ramifies the quaternion algebra associated to the $(d,0)$ surgery.

\printbibliography
\end{document}